\documentclass[11pt]{amsart}
\usepackage[top=2cm, bottom=4.5cm, left=2.5cm, right=2.5cm]{geometry}
\usepackage{amsmath,amsthm,amsfonts,amssymb,amscd}
\usepackage{enumerate}
\usepackage{mathrsfs}
\usepackage{xcolor}
\usepackage{graphicx}
\usepackage{tikz-cd}
\usepackage{multicol}
\usepackage{cancel}

\makeatletter
\def\@secnumfont{\bfseries}
\makeatletter

\numberwithin{equation}{section}
\numberwithin{equation}{subsection}

\theoremstyle{plain}

\newtheorem{theorem}[equation]{Theorem}
\newtheorem{lemma}[equation]{Lemma}
\newtheorem{proposition}[equation]{Proposition}

\newtheorem{corollary}[equation]{Corollary}

\newtheorem{cor}[equation]{Corollary}

\newtheorem{claim}[equation]{Claim}

\theoremstyle{definition}

\newtheorem{example}[equation]{Example}

\newtheorem{remark}[equation]{Remark}

\newtheorem{define}[equation]{Definition}

\newtheorem{nota}[equation]{Notation}
\newtheorem{ex}[equation]{Example}

\newtheorem{rem}[equation]{Remark}
\newtheorem{bekezdes}[equation]{}

\def\bC{\mathbb{C}}
\def\bR{\mathbb{R}}

\def\bZ{\mathbb{Z}}
\def\bH{\mathbb{H}}

\def\hh{\mathfrak{h}}
\def\vv{\mathfrak{v}}

\newcommand{\mm}{\mathbf{m}}
\newcommand{\ccc}{\mathbf{c}}
\newcommand{\elll}{p}
\newcommand{\calS}{\mathcal{S}}

\author{Alexander A. Kubasch}
\address{Alfr\'ed R\'enyi Institute of Math.,
Re\'altanoda utca 13-15, H-1053, Budapest, Hungary \newline
 \hspace*{3mm} ELTE - Univ. of Budapest, Dept. of Geo.,
 P\'azm\'any P\'eter s\'et\'any 1/A, 1117, Budapest, Hungary}
\email{kubasch.alexander@renyi.hu}

\author{Andr\'as N\'emethi}
\thanks{The authors are partially supported by  ``\'Elvonal (Frontier)'' Grant KKP 144148}
\address{Alfr\'ed R\'enyi Institute of Math.,
Re\'altanoda utca 13-15, H-1053, Budapest, Hungary \newline
 \hspace*{3mm} ELTE - Univ. of Budapest, Dept. of Geo.,
 P\'azm\'any P\'eter s\'et\'any 1/A, 1117, Budapest, Hungary \newline \hspace*{3mm}
  BBU - Babe\c{s}-Bolyai Univ., Str, M. Kog\u{a}lniceanu 1, 400084 Cluj-Napoca, Romania
 \newline \hspace*{3mm}
BCAM - Basque Center for Applied Math.,
Mazarredo, 14 E48009 Bilbao, Basque Country, Spain}
\email{nemethi.andras@renyi.hu }

\author{Gerg\H{o} Schefler}
\address{Alfr\'ed R\'enyi Institute of Math.,
Re\'altanoda utca 13-15, H-1053, Budapest, Hungary \newline
 \hspace*{3mm} ELTE - Univ. of Budapest, Dept. of Geo.,
 P\'azm\'any P\'eter s\'et\'any 1/A, 1117, Budapest, Hungary}
\email{schefler.gergo@renyi.hu}

\title
[Structural properties of the lattice cohomology of curve singularities]
{Structural properties of the lattice cohomology \\ of curve singularities}

\begin{document}

%\thanks{The second  author is partially supported by NKFIH Grant ``\'Elvonal (Frontier)'' KKP 126683.}

\begin{abstract}
    \noindent The lattice cohomology of a reduced  curve singularity was 
    introduced in \cite{AgNeIV}. It is a bigraded $\bZ[U]$-module $\bH^*=\oplus_{q,n}\bH^q_{2n}$,  that categorifies the $\delta$-invariant and extracts 
    key geometric information from the semigroup of values. 
    %and is conjectured to be functorial with respect to flat deformations.

     In the present paper we prove three  structure theorems for this new invariant: (a) the weight-grading of the reduced cohomology is \--- just as in the case of the topological lattice cohomology of normal surface singularities \cite{Nlattice} \--- nonpositive; %($\oplus _{n>0}\bH^*_{2n}=0$);
     (b) the graded $\bZ[U]$-module structure of $\bH^0$  determines whether or not a given curve is Gorenstein; and finally (c) the lattice cohomology module $\bH^0$ of any plane curve singularity determines its multiplicity.
\end{abstract}

\maketitle

\section{Introduction}

\subsection{}
The theory of curves has occupied a central place throughout the history of algebraic geometry, formulated productive programs, produced strong results, constructed deep invariants and
formidable bridges to other areas of mathematics. The theory was necessarily extended to the case of singular curves, thus creating the theory of local complex algebraic/complex analytic curve singularities.
In this later theory, complex plane curve singularities played a distinguished role. In the case of their study, in addition to the analytic/algebraic tools, low-dimensional topology also had an enormous contribution, and altogether they created 
a series of algebraic and topological invariants and deep results. Some recent developments have provided more impetus and new
invariants such as multivariate Poincar\'e series, the local punctual  Hilbert scheme,
new deformation invariants on the analytical side, Jones and HOMFLY polynomial, Heegaard Floer Link and Khovanov theory on the topological side.

On the other hand, for nonplanar curve singularities, topological tools can no longer be used, and
algebraic/analytical methods are much more difficult and less advanced too. In particular, the need for new invariants (which might even replace the missing topological invariants) is a real necessity.

Lattice cohomology is such a new invariant.  For reduced  curve singularities
it was introduced in \cite{AgNeIV}.
Since it has a rather rich structure, it  is able to encode  deep information about the curve germ. It is a bigraded $\bZ[U]$-module analogously to several cohomology theories in low dimesional topology, such as  the Heegaard Floer (HF) or  Heegaard Floer Link (HFL) homology. In fact,
 for {\it plane curve singularities} it can be related with the HFL homology
 (see \cite{GorNem2015,NFilt}), however, it is defined by analytic methods valid 
 to curve germs with arbitrarily large embedded dimension.
 Hence, it can be viewed as an algebraic analogue of the HFL theory
 valid for non-planar singularities as well. 
 %The hope is that it will become an important tool in the classification of germs.

 \subsection{}
 The lattice cohomology of reduced complex curve singularities is a  member of a family of lattice cohomology theories:
 one can define a topological lattice cohomology  \cite{Nlattice,NGr} and an analytic lattice cohomology \cite{AgNeI} associated to a normal surface singularity, and the analytic construction can be extended to any dimension \cite{AgNeHigh}.

 The structure of all lattice cohomologies is the following: first, $\bH^*$ is a bigraded $\bZ[U]$-module with a cohomological grading given by (the direct sum decomposition of $\bZ[U]$-modules) 
 $\bH^*=\oplus_{q\geq 0}\bH^q$ and a weight grading given by 
 $\bH^q=\oplus_n \bH^q_{2n}$ (where we use a $2\bZ$ grading because of the HFL parallelism). The $\bZ[U]$-module structure is given by a map $U:\bH^q_{2n}\to \bH^q_{2n-2}$ for every $q$ and $n$.

 In fact, $\bH^*$ appears as the cohomology of a certain filtered space $\cdots \subset S_n\subset S_{n+1} \subset\cdots$ as follows: $\bH^q_{2n} := H^q(S_n,\bZ)$ and the $U$-action is given by the maps induced by the inclusions $S_n \hookrightarrow S_{n+1}$. %In fact, the cohomology appears as a direct sum of singular cohomology groups. Namely,  $\bH^q=\oplus _nH^q(S_n, \bZ)$ for certain finite CW-complexes $\{S_n\}_n$ and $\bH^q_{2n}:=H^q(S_n,\bZ)$.The spaces $\{S_n\}_n$  form a tower (filtered space) $\cdots \subset S_n\subset S_{n+1} \subset\cdots$ and the $U$-action is the restriction map induced by these inclusions.
  The {\it reduced}
 lattice cohomology 
 (as a  bigraded $\bZ[U]$--module) 
 also  appears naturally: $\bH^*_{\rm red}=\oplus _q \bH^q_{\rm red}$, $\bH^q_{\rm red}=\oplus _n
 \widetilde{H}^q(S_n,\bZ)$.

 This filtered space satisfies the facts that $S_n=\emptyset $ for $n\ll 0$ and is contractible for $n\gg 0$, and it is a finite CW complex for any $n$.
 Therefore, 
 $\bH^*_{\rm red}$ has a finite $\bZ$-rank. It follows that $\bH^*$ is determined by  
 the smallest weight $\min\{n\,\vert \, S_n \not= \emptyset\}$ 
  (the analogue of the $d$-invariant used in HF homology \cite{OSzF})
 and 
 $\bH^*_{red}$. Using them 
 one also  defines the Euler characteristic by:
 \begin{equation}\label{eq:intro}
 {\rm eu}(\bH^*):= -\min\{n\,\vert \, S_n \not= \emptyset\}+\textstyle{\sum_q}\, (-1)^q\ {\rm rank}_\bZ\
 \bH^q_{\rm red}.\end{equation}

 \subsection{} In this article we study the lattice cohomology $\bH^*=\bH^*(C,o)$
 associated with a reduced complex curve singularity $(C,o)$. 
 In this case 
 the Euler characteristic of $\bH^*(C,o)$ is the delta invariant $\delta(C,o)$, that is,
  $\bH^*(C,o)$ is the categorification of $\delta(C,o)$ \cite{AgNeIV}.
  % (In the surface case,
  %${\rm eu}(\bH^*_{top})$ is the normalized canonical Seiberg-Witten invariant of the link,
  %while ${\rm eu}(\bH^*_{an})$ is the geometric genus of the normal surface singularity.)
  
  Our  goal here is to present three structure theorems for $\bH^*$ and  list certain properties of $\bH^*$ which
  identify concrete  families of curve singularities, or  provide new conceptual geometric information about distinguished 
  classical invariants. % E.g., natural questions are:
 %does  $\bH^*$ identify  the smooth curve germs or the Gorenstein ones?
 % Can we read off  the multiplicity of a curve  germ?
  The present note provides the first list of results in this direction.

 % This article contains three main structure theorems regarding $\bH^*(C,o)$ associated with a reduced curve singularity $(C,o)$.
   As we already mentioned the set  $\{n\,\vert \, \bH^*_{\text{red}, \, 2n}(C,o)\not=0\} $ is bounded both from above and below. The lower bound is an interesting new invariant of the curve singularity $(C,o)$, which will be not discussed in this note. On the other hand, the upper bound is universal:

  \begin{theorem}\label{th:int1} (see the Nonpositivity Theorem \ref{th:Main} and Theorem \ref{th:Main2})\label{thm:intr-nonpos}
    The spaces $\{S_n\}_{n \in \bZ}$ used to define $\bH^* (C,o)$ are contractible for all $n>0$. Thus, $\bH^q_{\text{\rm red}, \, 2n}(C,o) = 0$ for all $n>0, q\geq 0$.
\end{theorem}

 \noindent
 This statement leads to the following coarse classification of germs in terms of $\bH^0$,  based on how `far' the lower bound  is from the upper one, which also shows that the bound $n>0$ is optimal. 

  \begin{proposition} \label{thm:intr-class} (see Proposition \ref{thm:nonnegw})
  
  \noindent (a)  $\bH^0_{\rm red}=0$ if and only if $(C,o)$ is smooth, i.e. its multiplicity $mult(C,o)=1$.

  \noindent (b)  $\bH^0_{\rm red}\not=0$ but $\bH^0_{2n}=0$ for $n<0$
  (equivalently, $\min\{n|S_n\not=\emptyset\}\geq 0$)
  %is supported in weight zero ($\bH^0_{red}=\bH^0_{red,0}$) 
  if and only if  $mult(C,o)=2$. In this case $(C,o)$ is a plane curve singularity with a local equation of type $x^2+y^k=0$, $k\geq 2$.

  \noindent (c) $\bH^0_{2n}\not=0$ for some $n<0$ 
   (equivalently, $\min\{n|S_n\not=\emptyset\}<0$)
  if and only if  $mult(C,o)>2$.

\end{proposition}

The simplest germs with $\bH^0_{\rm red}\not =0$ (that is, when
 $\bH^0_{\rm red}=\bH^0_{{\rm red},0}\simeq \bZ$ and  $\bH^{\geq 1}=0$) are characterized in 
 Remark  \ref{rem:0locmin}.
 %E.g.  ordinary $r$-tuples  have these  properties (see Example \ref{ex:mult}). 
 %The non-smooth ordinary $r$-tuples 
 %have the additional property $r=2-\min \{w_0\}$ as well. 

 Regarding case {\it (c)} we have the next  characterizations of the Gorenstein property  and of $mult(C,o)$.

 \begin{theorem}\label{thm:mult-intr} If $mult(C,o)>2$ then the following hold.

\noindent (a) (Theorem \ref{th:Gorchar}) $(C,o)$ is Gorenstein if and only if
    $\text{\normalfont rank}_\bZ \ker\big( U : \bH^0_0(C,o) \to \bH^0_{-2}(C,o) \big) \geq 2$.

\noindent (b) (Theorem \ref{th:MF}) For plane curves
$mult(C,o)=2-\max\{n\,|\,  {\rm ker} (U:\bH^0_{2n}\to \bH^0_{2n-2}) \not =0,\ n<0\}.$
\end{theorem}

Part {\it (b)} is a generalization of \cite{KNS1}, where the identity 
was proved for irreducible 
plane germs. 

The main theorems (Theorem \ref{th:int1}, Theorem \ref{thm:mult-intr} {\it (a)} and
{\it (b)}) connect three geometric properties  with three different weight--zones of $\bH^*$: 
cohomological vanishing for $n>0$, Gorenstein property determined at weight level $n=0$  and multiplicity formula 
extracted from the $n<0$ zone. 

%\vspace{2mm}

The main tool of the note is the {\it weight function} of $(C,o)$ (see (\ref{eq:w0})).
  Not only that it is closely related with the classical invariants (it can be defined either from the   Hilbert function or from the semigroup of values, see subsection \ref{ss:2.2}), 
  but it is the new ingredient which leads to the definition of the lattice cohomology too.
  Basically, its level sets and critical points determine a new type of geometry of the singularity. The local minima of this weight function impose all
the results of the note. 
For a pictorial summary see Example \ref{ex:nagy}.

\subsection{} The structure of the paper is the following. In section 2 we recall the definition of the lattice cohomology and several other invariants of curve singularities (Hilbert function associated with normalization, semigroup of values, delta invariant, conductor, multiplicity)
 and we recall/establish certain properties of them. We introduce the weight function, whose level sets and critical points will play a crucial role in the subject. In section 3 we study the properties of the (generalized) local minima of the weight function and the corresponding local minimum values. 
 In section 4 we characterize the  Gorenstein curve germs by  several equivalent properties, one of them is  in terms of $\bH^0$. In section 5 we prove the Multiplicity Formula for any plane curve singularity.
 We provide several examples showing the subtlety of the
 problem in the non-planar case. 
 Finally, section 6 contains the proof of the contractibility of the spaces $\{S_n\}_{n>0}$.

\section{Invariants of a reduced  curve singularity} \label{s:Prem1}

In this section we give an overview of the \emph{lattice cohomology} and we also 
recall some classical
invariants of reduced  curve singularities.  We organize them
(by adding certain new properties too) as a preparatory part for the next sections. For more details about lattice cohomology  see \cite{AgNeIV,KNS1,NFilt}. For other versions  of lattice cohomologies in singularity theory see \cite{AgNeI,AgNeHigh,Nlattice,NBook}. For the Floer-theoretic interpretations 
in low-dimensional topology see e.g. \cite{GorNem2015,NFilt}.

\subsection{Definition of the lattice cohomology}\cite{AgNeIV} \label{latticedef} Let $(C,o)$ be a reduced  curve singularity with irreducible decomposition $(C,o) = \bigcup_{i\in\mathcal{I}}(C_i,o)$, where $\mathcal{I}=\{1,\dots,r\}$. The local algebra of $(C,o)$ and of a branch $(C_i,o)$ will be denoted by $\mathcal{O}=\mathcal{O}_{C,o}$ and $\mathcal{O}_i = \mathcal{O}_{C_i,o}$ respectively. Let $n_i : (\bC,0) \to (C_i,o)$ be the normalization map inducing the inclusion of rings $n_i^* : \mathcal{O}_i \hookrightarrow \overline{\mathcal{O}_i} = \bC\{t_i\}$ where $\overline{\mathcal{O}_i}$ denotes the integral closure of $\mathcal{O}_i$. Denote  the corresponding discrete valuations by
$$\vv_i :  \mathcal{O}\longrightarrow  \mathcal{O}_i \longrightarrow \bZ_{\geq 0}\cup \{ \infty \} \ \ \ ; \ \ \ g \longmapsto \text{ord}_{t_i}(n_i^*g).$$
They induce the $\bZ_{\geq 0}^r$-filtration $\mathcal{F}$ on $\mathcal{O}$ given by $\mathcal{F}(\ell) = \{ \, g \in \mathcal{O} \ | \ \vv_i(g) \geq \ell_i \text{ for all } i \in \mathcal{I} \, \}$, where $\ell = (\ell_1, \dots, \ell_r) \in \bZ_{\geq 0}^r$. Let  $\hh : \bZ_{\geq 0}^r \to \mathbb{Z}_{\geq 0}$ denote the Hilbert function given by the (finite) codimension $\hh(\ell) = \dim_\bC \mathcal{O} / \mathcal{F}(\ell)$ (see also \cite{CDG2}).
By definition, $\hh(\ell)\geq \hh(\ell')$ whenever $\ell\geq \ell'$.
Define the \emph{weight function} $w_0 : \bZ_{\geq 0}^r \to \bZ$ as
\begin{equation}\label{eq:w0}
w_0(\ell) = 2\cdot\hh(\ell) - |\ell|,\end{equation}
where $|\ell |= \sum_{i=1}^r\ell_i$. It carries the same information as $\hh$ (cf. (\ref{eq:hfromS}) and  (\ref{eq:w0valtozasa})) but via its level sets 
it describes and characterizes
more naturally several geometric phenomena.

The positive orthant $\bR_{\geq 0}^r $ has a natural decomposition into cubes. The
set of zero-dimensional cubes consists of the lattice points
$\bZ_{\geq 0}^r$. Any $\ell\in \bZ_{\geq 0}^r$ and subset
$I\subset \mathcal{I}$ of
cardinality $q$  defines a closed $q$-dimensional cube $(\ell, I)$, having as its
vertices the lattice points $\ell+\sum_{i\in K}e_i$, where
$K$ runs over all subsets of $I$, and 
 $e_i = (0,\dots,0,1,0,\dots,0) \in \bZ^r_{\geq 0}$ denotes the $i^{\text{th}}$ standard basis vector.
For closed cubes $(\ell, I)$ we use the notation $\square$ as well, open cubes will be denoted by $\square^{\circ}$.

%The set of $q$-dimensional cubes is denoted by $\mathcal{Q}_q$.
We define the weight of a 
%: \mathcal{Q}_0 = \bZ^r_{\geq 0} \to \bZ$ 
  $q$-cube ($q\geq 0$) by 
% : \mathcal{Q}_q \to \bZ$ via the formula
\begin{equation}\label{eq:w_qdef}
    w_q(\square) = \max\{\, w_0(\ell) \ | \ \ell \in \square \cap \bZ^r \, \} \ \
    \ \ (\text{for any $q$-cube $\square$}),
\end{equation}
and  the `sublevel sets' $S_n$ of the weight function by 
%Let $x \in \bR^r_{\geq 0}$ be a point. There exists a unique minimal (with respect to containment)
%cube $\square_x \in \mathcal{Q}_q$ containing the point $x$.
%From definitions  $w(x) = w_q(\square_x)$, where $w$ is the extended weight function defined in 
%subsection \ref{latticedef}.
%
%\begin{cor}\label{cor:2.5.2}  Let $(C,o)$ be a reduced curve singularity. Then
%    the sublevel set $S_n$ of the extended weight function $w$ is in fact the CW complex
\begin{equation*}
    S_n = \bigcup\{ \, \square \ | \ \square \ \mbox{is a $q$-cube for any $q\geq 0$ with}
    \ w_q(\square) \leq n \, \}.
\end{equation*}
%\end{cor}

 %The weight function $w_0$ is extended to the real positive orthant $\bR_{\geq 0}^r$ by defining
% $w(x)$ as $ \max \{ \, w_0(\ell)\} $,  where $\ell$  runs over all possible lattice points such 
%that its coordinate
%$\ell_i$ is either $\lfloor x_i \rfloor$   or  $\lceil x_i \rceil$
%for all  $i$. The sublevel sets of $w$ are denoted by
%$$S_n = \{ \, x \in \bR_{\geq 0}^r \ | \ w_0(x) \leq n \, \}.$$
 
 It turns out that $w_0$ is bounded from below and $\{\ell\in \bZ_{\geq 0}^r\ |\ w_0(\ell)\leq n\} $
is finite for every $n \in \bZ$. Hence $S_n$ is
 a finite closed cubical complex and  it is empty  for $n\ll 0$. 
 %(see also  \ref{latticedef}).
 %and Theorem \ref{th:EUcurves}.
 Moreover $S_n  \subset S_{n+1}$. 

\begin{define}
    The lattice cohomology of the reduced 
    curve singularity $(C,o)$ is defined to be
    $$\bH^*(C,o) = \bigoplus_{n \in \bZ}H^*(S_n,\bZ).$$
    It is a $(\bZ \times 2\bZ)$-graded $\bZ[U]$-module. The first (homological) grading is $\bH^*=\oplus _q\bH^q$, while the second (weight) grading $\bH^q=\oplus_{n}\bH^q_{2n}$ is defined as $\bH^q_{2n}(C,o) = H^q(S_n,\bZ)$. The  $U$--action is the  morphism  $H^q(S_{n+1},\bZ) \to H^q(S_{n},\bZ)$ induced by the inclusion. It is homogeneous of degree $(0,-2)$.

    The \emph{reduced lattice cohomology} $\bH^*_{\text{red}}(C,o)$ is also defined by setting $\bH^q_{\text{red}, \, 2n}(C,o) = \widetilde{H}^q(S_n,\bZ)$,  where $\widetilde{H}^*$ denotes reduced singular cohomology.
\end{define}

\begin{rem}\label{rem:lhom}   (a)
    We prefer the convention of using this $(\bZ \times 2\bZ)$-grading instead of the $\bZ^2$-grading in order to keep the compatibility with the Heegaard-Floer (Link) theory, see \cite{NFilt}.
%\end{rem}

%\begin{rem}
 (b)   By another construction, $\bH^*(C,o)$ arises as the homology of a  cochain complex associated with the weight function and the  cubical decomposition described 
 %in subsection \ref{latticedef}, 
 above, see \cite{AgNeIV}.
%\end{rem}

(c) From the very construction it follows that $\bH^{\geq r}(C,o)=0$.

%\begin{rem}
 (d)   In the case of a complex {\it plane} curve singularity $(C,o) \subset (\bC^2,0)$, the Hilbert function can be recovered from the multivariate Alexander polynomial of the link $C\cap S^3_{\varepsilon} \subset S^3_\varepsilon $ \cite{GorNem2015} (see also \cite{yamamoto}). It follows that in this planar  case $\bH^*(C,o)$ is an embedded topological invariant.
\end{rem}

\begin{define}\label{bek:gradedroot}
The graded $\bZ[U]$-module $\bH^0(C,o)$ has an improvement, which is also a convenient pictorial presentation.
It is the {\it graded root} $\mathfrak{R}(C,o)$, a $\bZ$-graded connected graph. The vertices $\mathfrak{V}$ are $\bZ$-graded,  $\mathfrak{V}=\sqcup _n
\mathfrak{V}_n$, the vertices  $\{v_{n,i}\}_i\in \mathfrak{V}_n$ correspond bijectively to the connected components $\{S_{n,i}\}_i$ of $S_n$. The edges are defined as follows. If two components $S_{n,i}$ and $S_{n+1,j}$
of $S_n$ and $S_{n+1}$  satisfy $S_{n,i}\subset S_{n+1,j}$ then we connect $v_{n,i}$ and $v_{n+1,j}$ by an edge. %These are all the edges.
From $\mathfrak{R}(C,o)$ one can read  $\bH^0(C,o)$: each vertex of degree $n$
generates a free $\bZ$-summand in $\bH^0_{2n}=H^0(S_n,\bZ)\cong\oplus_i H^0(S_{n, i},\bZ)$, and the $U$-action is marked by the edges.
For more see e.g. \cite{NGr,Nlattice,NBook}. % (or several examples of this note).
\end{define}

\subsection{Connection to the semigroup of values}\label{ss:2.2} \cite{AgNeIV,delaMata87,delaMata,Garcia} The \emph{semigroup of values} of the reduced 
 curve singularity $(C,o) = \bigcup_{i\in\mathcal{I}}(C_i,o)$ is defined to be
\begin{equation}\label{eq:S_Cdef}
    \mathcal{S}_{C,o} = \big\{ \, \big(\vv_1(g), \dots, \vv_r(g)\big)  \in \bZ_{\geq 0}^r \ \big| \ g \in \mathcal{O} \text{ is not a zero-divisor} \, \big\}.
\end{equation}
 The Hilbert function $\hh$ determines the semigroup $\mathcal{S}_{C,o}$ as follows:
\begin{equation}\label{eq:Sfromh}
  \mathcal{S}_{C,o}=\{\, \ell\ \in \bZ_{\geq 0}^r \ | \  \hh(\ell+e_i)>\hh(\ell) \text{ for all } i \in \mathcal{I} \, \}.
 \end{equation}
%\begin{rem} \label{rem:minclosed}
% It is known (use e.g. (\ref{eq:Sfromh}) and (\ref{eq:hfromS}))
% that $\mathcal{S}_{C,o}$ is closed under taking minima: if $s, s' \in \mathcal{S}_{C,o}$, then $\min\{ s,s'\} \in \mathcal{S}_{C,o}$ too, where min denotes the coordinatewise minimum.
% \end{rem}

\begin{define}\label{def:Delta} \cite{delaMata87,delaMata,Garcia}
    Given the semigroup of values $\mathcal{S}_{C,o} \subset \bZ^r_{ \geq 0}$ 
    and $\ell \in \bZ^r$
    we set
    $$\Delta_i(\ell) := \{\, s \in \mathcal{S}_{C,o} \ | \ s_i=\ell_i  \text{ and } s_j > \ell_j \text{ for all } j \not= i \, \}$$
    and $\Delta(\ell) := \bigcup_{i \in \mathcal{I}} \Delta_i(\ell).$
    Similarly, let us define the `closures' of the above sets as
    $$\overline{\Delta}_i(\ell) := \{\, s \in \mathcal{S}_{C,o} \ | \ s_i= \ell_i \text{ and } s_j \geq \ell_j \text{ for all } j \not= i\, \}$$
    and $\overline{\Delta}(\ell) := \bigcup_{i \in \mathcal{I}} \overline{\Delta}_i(\ell).$
\end{define}
%Note that $\Delta_i(\ell)= \overline{\Delta}_i(\ell+\mathbf{1}-e_i)$, where $\mathbf{1}=(1,\ldots, 1)$.
Conversely to (\ref{eq:Sfromh}), the semigroup $\mathcal{S}_{C,o}$
also determines the Hilbert function $\ell\mapsto \hh(\ell)$ as follows:
\begin{equation}\label{eq:hfromS}
\hh(\mathbf{0})=0 \ \ \ \text{and} \ \ \ \hh(\ell+e_i) - \hh(\ell) = \begin{cases}
    \ 1 \ \text{ if and only if } \ \overline{\Delta}_i(\ell) \not= \emptyset  \\
    \ 0 \ \text{ if and only if } \ \overline{\Delta}_i(\ell) = \emptyset
\end{cases}
  \end{equation}
for all $i \in \mathcal{I}$ and $\ell\geq \mathbf{0}$. In particular, if $(C,o)$ is irreducible, then $ \mathfrak{h}(\ell) = \#\{ \, s < \ell\ | \ s \in \mathcal{S}_{C,o} \, \}$.
%for all $\ell \in \bZ_{\geq 0}$ and $\delta(C,o)=\#\{\bZ_{\geq 0}\setminus \mathcal{S}_{C,o}\}$.

One may reformulate these facts using the weight function $w_0 : \bZ^r_{\geq 0} \to \bZ$ as follows:

\begin{cor}\label{cor:wfromS} Let $w_0$ be the weight function defined
in (\ref{eq:w0}). Then
    $w_0(\mathbf{0}) = 0$ and
    \begin{equation}\label{eq:w0valtozasa}
w_0(\ell+e_i) - w_0(\ell) = \begin{cases}
    \ +1 \ \text{ if and only if } \ \overline{\Delta}_i(\ell) \not= \emptyset  \\
    \ -1 \ \text{ if and only if } \ \overline{\Delta}_i(\ell) = \emptyset
\end{cases}
  \end{equation}
  for all $i \in \mathcal{I}$  and $\ell\geq \mathbf{0}$. In particular, if $(C,o)$ is an irreducible curve singularity, then
  \begin{equation}\label{eq:w_0irredre}
      w_0(\ell) = \#\{ \, s < \ell \ | \ s \in \mathcal{S}_{C,o} \, \} - \#\{ \, 0 \leq k < \ell \ | \ k \not\in \mathcal{S}_{C,o} \, \} \ \ \ \mbox{ for all $\ell \in \bZ_{\geq 0}$.}
  \end{equation}
%  for all $\ell \in \bZ_{\geq 0}$.
\end{cor}

\begin{nota}\label{def:rect}
    Given two lattice points $a,b \in \bZ^r$, let $$R(a,b):=\{ \, x \in \bR^r \ | \ a_i \leq x_i \leq b_i \text{ for all } i\in\mathcal{I} \, \}.$$
\end{nota}

\begin{rem} \label{rem:h|R(0,E)}
    Let $\mathbf{1} := \sum_{i=1}^re_i$. If $s\in \mathcal{S}_{C, o}\setminus \{  \mathbf{0}\}$ then necessarily $s\geq \mathbf{1}$ (compare also with \ref{rem:semigroup}).
    %Since $\mathbf{0} \in \mathcal{S}_{C,o}\cap R(\mathbf{0},\mathbf{1}) \subseteq \{\mathbf{0},\mathbf{1}\}$, it follows that
    Hence $\mathfrak{h}(\ell) = 1 $ for all $ \ell \in R(\mathbf{0}, \mathbf{1}) \cap \mathbb{Z}^{r} \setminus \{ \mathbf{0}\}$.
% In particular $\hh(e_i)=1$ for all $i \in \mathcal{I}$.
\end{rem}
%For some additional properties of $\mathcal{S}_{C,o}$ see subsection \ref{rem:semigroup}.

\subsection{The conductor and the $\delta$-invariant}\label{ss:cond} \cite{Zar} Assume that $(C,o)$ is not smooth. Let ${\mathfrak c}=(\mathcal{O}:\overline{\mathcal{O}})$ be the {\it conductor ideal} of the normalization $\overline{\mathcal{O}}=\bigoplus_{i=1}^r\bC\{t_i\}$ of $\mathcal{O}$, i.e. the largest (proper) ideal of $\mathcal{O}$ that is also an ideal of $\overline{\mathcal{O}}$. It has the form
$(t_1^{c_1}, \ldots, t_r^{c_r})\overline{\mathcal{O}} $. The lattice point $\mathbf{c}=(c_1,\ldots , c_r) \in \bZ_{\geq 0}^r$ is called the
{\it conductor} of $\mathcal{S}_{C,o}$. It is the smallest lattice point (with respect to the coordinate-wise partial ordering) with the property $\mathbf{c}+\mathbb{Z}_{\geq 0}^r\subset \mathcal{S}_{C,o}$.
 In particular,  $\mathcal{F}(\mathbf{c}) = \mathfrak{c}$ (for $\mathcal{F}$ see \ref{latticedef}). 

\begin{define} The {\it delta invariant } of $(C,o)$ is defined as $\delta(C,o):=\dim_{\bC}\overline{\mathcal{O}_{C,o}}/\mathcal{O}_{C,o}$. 
\end{define}
     The delta invariant  is zero if and only if $(C,o)$ is smooth.
    It is connected to the conductor via the following facts: $\dim \overline{\mathcal{O}}/\mathfrak{c}=|\mathbf{c}|$
    and  $\dim \mathcal{O}/\mathfrak{c}=\hh(\mathbf{c})$, hence
    $\delta(C,o)=|\mathbf{c}|-\hh(\mathbf{c})$.
%\end{define}

%Thus, $\delta(C,o)\leq |\mathbf{c}|-1$,
%with equality exactly when $\mathfrak{c}$ is the maximal ideal of
%$\mathcal{O}$. One also has $\delta(C,o)\geq|\mathbf{c}|/2 $ 
%with equality if and only if
%$(C,o)$ is Gorenstein, cf. \cite[p. 72]{Serre} (see also  Theorem
%\ref{th:Gorchar} here).

\begin{theorem}\label{th:EUcurves} \cite{AgNeIV} For any lattice point $\ell\geq \mathbf{c}$ the inclusion $S_n\cap R(\mathbf{0},\ell)\hookrightarrow S_n$ is a homotopy equivalence. In particular, $S_n$ is empty for $n \ll 0$ and contractible for $n
\geq \max \{w_0 |_{R(\mathbf{0}, \ccc )}\}$. 

Therefore, for any lattice point $\ell \geq \mathbf{c}$ there exists a canonical bigraded $\bZ[U]$-module isomorphism between $\bH^*(C,o)$ and $\bigoplus_{n \in \bZ} H^*\big(S_n \cap R(\mathbf{0},\ell),\bZ\big)$.
(Similarly, $\mathfrak{R}(C,o)$ can also be reconstructed from the connected components of
$\{S_n \cap R(\mathbf{0},\ell)\}_n$ via the usual construction from \ref{bek:gradedroot}.)
\end{theorem}
\begin{remark}  \label{th:EuChar} (a)
By Theorem \ref{th:EUcurves} 
each $\bH^q_{\text{\normalfont red} }(C,o)$ has finite $\bZ$-rank.
Recall also that $\bH^q_{\text{\normalfont red} }(C,o)=0$ for $q\geq r$. 
This allows us to define the Euler characteristic ${\rm eu}(\bH^*)$ of $\bH^*$ (see (\ref{eq:intro})). 

(b) By \cite{AgNeIV} ${\rm eu}(\bH^*)=\delta(C,o)$.
%equals the delta invariant of $(C,o)$  
%\end{remark}

%\begin{define} \cite{AgNeIV} The Euler characteristic of $\bH^*(C,o)$ is defined as \begin{equation*} \text{eu}\big(\mathbb{H}^*(C,o)\big)=-\min\{w_0\} +\sum_{q\in \bZ} (-1)^q{\rm rank}_\bZ\big(\mathbb{H}^q_{\text{red}}(C,o)\big). \end{equation*}
%Note that this makes sense, as by Theorem \ref{th:EUcurves} $S_n$ is a finite CW complex whence
 %   $\text{\normalfont rank}_\bZ \big(\bH^*_{\text{\normalfont red} }(C,o)\big)$ is finite, too.\end{define}

%\begin{theorem} \cite{AgNeIV} \label{th:EuChar} $\text{eu}\big(\mathbb{H}^*(C,o)\big)=\delta(C,o)$.
 %   The Euler characteristic of the lattice cohomology $\bH^*(C,o)$ is equal to the delta invariant $\delta(C,o) = \dim_\bC\overline{\mathcal{O}}/\mathcal{O}$. I.e., $\bH^*(C,o)$ is a `categorification' of $\delta(C,o)$.\end{theorem}

%\begin{example}\label{ex:CDparospelda3}
(c) According to  Theorem \ref{th:Main2},
if $\bH^{\geq 1}=0$ then ${\rm eu}(\bH^*)=\delta$ can be computed from the graded root $\mathfrak{R}$  as follows (this can be useful in the examples presented below).
 Delete all the vertices $v$ of the root 
with $w_0(v)>0$.  Then ${\rm eu}(\bH^*)+1$ is the number of remaining vertices.
\end{remark}

\begin{example}\label{ex:decomp} ({\bf Decomposable germs})
A curve singularity $(C,o)$ is called `decomposable' (into
$(C',o)$ and $(C'',o)$),
 if it is isomorphic to the one-point union in
$(\bC^n\times \bC^m,(0,0))$ of  $(C',o)\times \{o\}\subset (\bC^n,0)\times \{o\}$ and
$\{o\}\times (C'',o)\subset \{o\}\times (\bC^m,0)$.
We denote this by $(C,o)=(C',o)\vee (C'',o)$ \cite{Stevens85}.

%If $(C',o)\subset (\bC^n,0)$ is given by the ideal $I'$, then its ideal in $(\bC^n\times \bC^m,(0,0))$ is $I'+\mathfrak{m}_{(\bC^m,o)}$. %(here $\mathfrak{m}$ denotes the maximal ideal).
%Hence,  if $(C,o)=(C',o)\cup(C'',o)$ then $(C,o)$ is decomposable into  $(C',o)$ and $(C'',o)$ if and only if $(C',C'')_{Hir}=1$,
%(or, if and only if $\delta(C,o)=\delta(C',o)+\delta(C'',o)+1$, cf. (\ref{eq:Hir})),
%see \cite{Steiner83,Stevens85}. 
Let  $r'$ and $r''$ be the number of
irreducible components of  $(C',o)$ and $(C'',o)$. Then, 
by \cite[6.3]{NFilt},
the semigroup of values of a decomposable germ satisfies
\begin{equation}\label{eq:dec1}
\calS_{C,o}=\{(\mathbf{0},\mathbf{0})\}\cup\, \big( (\calS_{C',o}\setminus \{\mathbf{0}\})\times (\calS_{C'',o}\setminus \{\mathbf{0}\} )\big)
\subset
\bZ_{\geq 0}^{r'}\times \bZ_{\geq 0}^{r''}= \bZ_{\geq 0}^{r'+r''},
\end{equation}
hence $\mathbf {c}(C,o)=(\mathbf{c}(C',o), \mathbf{c}(C'',o))$, and
%In particular,
%Therefore, by Lemma \ref{eq:semi} and from  $w(l)=2\hh(l)-|l|$   we also have
%\begin{equation}\label{eq:dec2}
 %\hh_{(C,o)}(l',l'')=\left\{
 %\begin{array}{l}
 %       \hh_{(C',o)}(l') \\%& \mbox{if $l''=0$},\\
 %         \hh_{(C'',o)}(l'')\\%& \mbox{if $l'=0$},\\
 %           \hh_{(C',o)}(l') +\hh_{(C'',o)}(l'')-1 \\%& \mbox{if $l'>0$ and $l''>0$};
 %      \end{array}\right.
 %       w_{(C,o)}(l',l'')=\left\{
% \begin{array}{ll}
%        w_{(C',o)}(l') & \mbox{if $l''=0$},\\
%          w_{(C'',o)}(l'') & \mbox{if $l'=0$},\\
%            w_{(C',o)}(l') +w_{(C'',o)}(l'')-2 & \mbox{if $l'>0$ and $l''>0$};
%       \end{array}\right.
%\end{equation}
\begin{equation}\label{eq:dec3}
 w^C_0(\ell',\ell'')=\left\{
 \begin{array}{ll}
        w_0^{C'}(\ell') & \ \ \mbox{if $\ell''=0$},\\
          w_0^{C''}(\ell'') & \ \ \mbox{if $\ell'=0$},\\
            w_0^{C'}(\ell') +w_0^{C''}(\ell'')-2 & \ \ \mbox{if $\ell'>0$ and $\ell''>0$}.
       \end{array}\right.
\end{equation}

\end{example}

\subsection{Increasing paths and a new  characterization of the conductor}

\begin{define}
    An \textit{increasing path} starting from $a\in \bZ^r_{\geq 0}$
     is a sequence of lattice points $\gamma = \{ x^k\}_{k=0}^t$ such that
      $x^0= a$, $x^{k+1}=x^k+e_{i(k)}$ for some index $i(k) \in \mathcal{I}$.
      If $t$ is finite then we say that the path  connects $a$ with $x^t$ (or it is an increasing path from $a$ to $x^t$).
      If $t$ is infinite then we require for the infinite increasing path
      to be unbounded in every coordinate.
%      that all the coefficients of the lattice points along the path to tend to infinity.
\end{define}

The next Lemma is the converse of the obvious fact that for any $a\geq \mathbf{c}$
there exists an infinite increasing path starting from $a$ and having all its elements in $\mathcal{S}_{C,o}$.

\begin{lemma}\label{lem:chcond} Assume that there exists an infinite increasing path
$\{x^k\}_{k\geq 0}$ such that
$x^k\in \mathcal{S}_{C,o}$ for every $k\geq 0$. Then $x^0\geq \mathbf{c}$.
In particular, 
 the conductor $\ccc$ is the smallest lattice point such that for any $a \geq \ccc$ there exists an infinite increasing path $\gamma \subset \mathcal{S}_{C,o}$ starting at $a$.
\end{lemma}
\begin{proof}
Fix an arbitrary lattice point $\ell\geq x^0$. We want to prove that $\ell$ is a semigroup element by checking that $w_0(\ell+e_{i})>w_0(\ell)$ for any $i \in \mathcal{I}$ (cf. (\ref{eq:Sfromh})). For any fixed coordinate $i$ we can choose a finite  increasing path
$\{ y^k\}_{k=0}^t$ (where $t$ is sufficiently large  depending on  $\ell$ and $i$)
such that $y^0=x^0$,   $y^t=x^t$, and for some index $0 \leq j < t$ we have $y^j=\ell$ and $y^{j+1}=\ell+ e_i$ (in fact, $j = |\ell-x^0|$ and then $i(j)=i$).
Since  $w_0$ is increasing along $\{ x^k\}_{k=0}^t$ (cf. (\ref{eq:w0valtozasa})), 
$w_0(y^t)-w_0(y^0)=w_0(x^t)-w_0(x^0)=t$. Moreover, since  $\{ y^k\}_{k=0}^t$ has length $t$ too, $w_0$ should be  increasing along it as well, hence $w_0(\ell+e_{i})> w_0(\ell)$.
Thus, we got that for any $\ell \geq x^0$ we have $\ell \in S_{C,o}$ and hence $x^0 \geq \ccc$.
\end{proof}

\subsection{Stability properties}\cite{AgNeIV} 
%Let $(C,o) = \bigcup_{i\in\mathcal{I}}(C_i,o)$ be a reduced curve singularity.
The fact, that the Hilbert function $\hh$ is defined by a valuative filtration $\mathcal{F}$ (cf. subsection \ref{latticedef}),
implies the \emph{`matroid rank inequality'}:
 \begin{equation}\label{eq:matroid}
 \hh(\ell_1)+\hh(\ell_2)\geq \hh\big(\min\{\ell_1,\ell_2\}\big)+\hh\big(\max\{\ell_1,\ell_2\}\big)
 \end{equation}
for all lattice points $\ell_1, \ell_2 \in \bZ^r_{\geq 0}$. This  can be translated to a   \emph{`stability property'} of $\hh$, which in terms of the 
weight function $w_0 : \bZ^r_{\geq 0} \to \bZ$ reads as follows:
\begin{cor}\label{cor:stabforw0}
For any $\ell, \bar{\ell} \in \bZ^r_{\geq 0}$ such that $\bar{ \ell}\geq 0$ and $i \not\in {\rm supp}\big(\bar{\ell}\big):=\{j\ | \ \bar{\ell}_j\not=0\}$ (i.e. $\bar{\ell}_i=0$)
 \begin{equation}\label{eq:genstability-}
 w_0(\ell+e_i)=w_0(\ell)-1\ \ \Rightarrow\  \ w_0(\ell+\bar{\ell}+e_i)=w_0(\ell+\bar{\ell})-1.
 \end{equation}
Equivalently,
\begin{equation}\label{eq:genstability+}
 w_0(\ell+e_i)=w_0(\ell)+1\ \ \Rightarrow\  \ w_0(\ell-\bar{\ell}+e_i)=w_0(\ell-\bar{\ell})+1.
 \end{equation}
\end{cor}

\subsection{Restricting the weight function}\label{bek:w0megszoritas}
Let $(C,o) \subset (\mathbb{C}^N,0)$ denote a reduced  curve singularity with irreducible branches $(C_i,o)$ for $i \in \mathcal{I}= \{1,\dots, r\}$. Given any subset $I \subseteq \mathcal{I}$ denote by $(C_I,o)$ the singularity having irreducible branches $(C_i,o)$ with $i \in I$. Let us denote the  weight function corresponding to this singularity by $w_0^{C_I}:\bZ^{\#I}_{\geq 0} \rightarrow \bZ$.  $\bZ^{\#I}$ is a sublattice of $\bZ^r$ in a natural way.

The following lemma is well-known (it follows from the definitions,
or one can use  (\ref{eq:hfromS}) as well).

\begin{proposition}\label{lem:wrestricted}
    Let $J \subset I \subset \mathcal{I}$ be two subsets and $\ell = \sum_{j \in J} \ell_j e_j \in \bZ^r_{\geq 0}$ a lattice point with support in $ J$ (i.e. $\ell_j=0$ if $j\not\in J$).
    Then $\hh^{C_I}(\ell) = \hh^{C_J}(\ell)$, hence $w_0^{C_I}(\ell) = w_0^{C_J}(\ell)$ as well.
\end{proposition}

\subsection{The multiplicity of $(C,o)$} \label{ss:multlocmin} Let $(C,o)$ be a reduced  curve singularity as above. Let $\mathfrak{m}_{C,o}$ denote
 the maximal ideal of $\mathcal {O}_{C,o}$.  The multiplicity $m=m(C,o)$ is defined as the degree
of a generic linear projection $L:(C,o)\to (\mathbb{C},0)$ (i.e.
$L\in \mathfrak{m}_{C,o}$ is generic).
In particular,  $m = \sum_{i \in \mathcal{I}}m_i$, where $m_i = m(C_i,o)$.
Equivalently, if $ \dim_{\bC} (\mathcal O_{C,o}/\mathfrak{m}_{C,o}^k)=ak+b$ ($k\gg 0$)
is the Hilbert--Samuel  polynomial of $(C,o)$, then $m(C,o)$ is the coefficient $a$ of $k$.
 
  Set the `multiplicity vector'
$$\mm = (m_1,\dots,m_r) \in \bZ_{\geq 0}^r.$$
It is the image of a generic linear function under the valuation map $(\vv_1,\dots,\vv_r) : \mathcal{O}_{C,o} \to( \bZ_{\geq 0}\cup\infty)^r$ and hence the smallest non-zero element of the semigroup $\mathcal{S}_{C,o}$.
 This fact (together with properties (1)-(2) from subsection \ref{rem:semigroup})
 shows that  $\mathcal{S}_{C,o} \setminus R(\mm, \infty) = \{\mathbf{0}\}$. Thus (using Corollary \ref{cor:wfromS}) we have
\begin{equation}\label{eq:mineq}w_0(\mm) = 2-m(C,o).\end{equation}
Clearly, $m=1$ if and only if $(C,o)$ is smooth, and this happens exactly when
$\bH^*_{\text{\normalfont red}}(C,o)=0$, see e.g. \cite[Example 4.6.1]{AgNeIV}.
The case $m=2$ can be classified as well.
By Abhyankar's inequality \cite{Ab}
\begin{equation}\label{eq:Ab}
\dim(\mathfrak{m}_{C,o}/\mathfrak{m}_{C,o}^2)=\text{ embdim}(C,o)\leq m(C,o).
\end{equation}
In particular, if $m(C,o)=2$, then  the embedding dimension $\text{embdim}(C,o)$
is necessarily 2, hence $(C,o)$ is a plane curve singularity with multiplicity 2, i.e. in some local coordinates it has equation $x^2+y^k$, $k\geq 2$, whose weight function $w_0$ is easily computable and it satisfies $w_0\geq 0$. 

These facts together with (\ref{eq:mineq}) imply the following

\begin{proposition}\label{thm:nonnegw}
    For a reduced  curve singularity the following facts hold:

    (a)  $m(C,o) \leq 2$    if and only if
$\bH^*_{2n}(C,o)=0$ for all $n<0$ (i.e. 
      $w_0(\ell) \geq 0$ for all $\ell \in \bZ^r_{\geq 0}$),

     (b)
      $m(C,o) = 1$    if and only if
     $\bH^0_{\text{\normalfont red}}(C,o)=0$,

     (c)
      $m(C,o) = 2$    if and only if
     $\bH^0_{\text{\normalfont red}}(C,o)\not=0$, but $\bH^0_{2n}(C,o) = 0$ for all $n < 0$.

     \vspace{2mm}

     \noindent
     In all these cases (by a direct computation)  $\bH^{>0}(C,o)=0.$
\end{proposition}

\noindent
This already shows that the graded module $\bH^0(C,o)$ determines whether $m(C,o)$ is $1$, $2$, or $\geq 3$.

%\begin{example}\label{ex:mult}({\bf Ordinary $r$-tuples)} Using Hironaka's formula  for the delta invariant one sees that $\delta(C,o)\geq r-1$ ($r$ being the  number of irreducible components) \cite{Hironaka, Stevens85}. The optimal bound is realized by {\it ordinary $r$-tuples}, introduced in \cite{Stevens85} (see also \cite{BG80});  they can also  be defined as those germs which are isomorphic with the union of the $r$ coordinate axes in $(\bC^r,0)$, i.e., $(C,o)=\vee_{i=1}^r (\bC,0)$.  From our discussions the following equivalent facts hold for non-smooth germs: \vspace{2mm}

%(i) $(C,o)$ is an ordinary $r$-tuple for some $r\geq 2$, 

%(ii) ${\mathbf c}={\mathbf 1}$,

%(iii)  $\bH^0_{{\rm red}}=\bH^0_{{\rm red},0}\simeq \bZ$, $\bH^{\geq 1}=0$
%and $2-\min \{w_0\}=\{\mbox{number of irreducible components}\}$.

%For 
%(i)$\Leftrightarrow$(ii) see \cite{BG80,Stevens85},  (i)$\Rightarrow$(ii) follows from (\ref{eq:dec1}), (ii)$\Rightarrow$(iii) from Remark \ref{rem:h|R(0,E)} and Theorem \ref{th:EUcurves}, and (iii)$\Rightarrow$(i) from Remark \ref{th:EuChar}. \end{example}

\begin{remark}\label{rem:0locmin}
It is natural to ask for the characterization of all
those germs with $\bH^0_{{\rm red}}\not=0$, but with simplest possible $\bH^0_{{\rm red}}$.

If  $\bH^0_{{\rm red}}\not=0$ then $S_0$ has at least two connected components, one of them 
consists of the single lattice point $\mathbf {0}$, an other one contains $\mathbf {1}$.
This follows from the structure of $\hh$-values in $R(\mathbf{0},\mathbf{1})$, cf. Remark \ref{rem:h|R(0,E)}. 
In particular, the simplest case when $\bH^0_{{\rm red}}\not=0$ is
$\bH^0_{{\rm red}}=\bH^0_{{\rm red},0}\simeq \bZ$, $\bH^{\geq 1}=0$.
These conditions can also be characterized either by
the identity ${\mathbf c}={\mathbf m}$, or, equivalently,  in terms of the semigroup, by
$\mathcal {S}_{C,o}=\{\mathbf{0}\}\cup (\mathbf{c}+\bZ^r_{\geq 0})$ (See Example \ref{ex:c}).
A non-smooth ordinary
$r$-tuple (\cite{Stevens85,BG80})
is a particular case satisfying additionally ${\bf c}={\bf m}={\bf 1}$ and  $r=2-\min \{w_0\}$.

\end{remark}
\subsection{Good semigroups}\label{rem:semigroup}
%begin{remark} \ 

(a) \
Recall that the (sub)semigroup of values $\mathcal{S}_{C,o}\subset \bZ^r_{\geq 0}$
has the special property that it admits a conductor: there exists
$\mathbf{c}\in \bZ^r_{\geq 0}$ such that $\mathbf{c}+\bZ^r_{\geq 0}\subset\mathcal{S}_{C,o}$.

If $r=1$, then any numerical semigroup $\mathcal{S}\subset \bZ_{\geq 0}$ (with $0\in\mathcal{S}$)
which admits a conductor can be realized as $\mathcal {S}_{C,o}$ for
a certain germ $(C,o)$: take e.g.
$(C,o)=({\rm Spec}\,\bC[\mathcal{S}],0)$.

However, a similar statement does not hold for any  subsemigroup of $\bZ^r_{\geq 0}$  with $r>1$, even the uniqueness  of $\mathbf{c}$ might fail (and the $\bC$-algebra
$\bC[\mathcal{S}]$ might not be finitely generated). In the $r>1$ case the literature
lists several additional properties which are satisfied by any semigroup of values $\mathcal{S}=\mathcal{S}_{C,o}$ (the validity of the next  properties (1)-(4)
is rather straightforward, see e.g. \cite[p. 350]{delaMata87}):

\vspace{2mm}

(1) $\mathbf{0}\in  \mathcal{S}$, moreover,
 if $s\in\mathcal{S}$ and  $s_i=0$ for some $i$ then $s=\mathbf{0}$;

(2) if $s',\, s''\in\mathcal {S}$ then $\min\{s',s''\}\in\mathcal{S}$;

(3) if $s',\, s''\in\mathcal {S}$ and $s'_i=s''_i$ for some $i$, then there exists
$t\in\mathcal{S}$, $t\geq \min\{s',s''\}$ such that $t_i>s_i'=s''_i$ and
$t_j=\min\{s'_j,s''_j\}$ if  $j\not=i$ and $s'_j\not=s''_j$;

(4) $\mathcal{S}$ admits a conductor (via (2), in fact, there exists a unique minimal
$\mathbf {c}$ with  $\mathbf{c}+\bZ^r_{\geq 0}\subset\mathcal{S}_{C,o}$).

\vspace{2mm}

A subsemigroup $\mathcal{S}\subset \bZ^r_{\geq 0}$ with these properties is called
{\it good (local)}, see e.g. \cite{good2,good}. % and the references therein.

(b) \ Not any good semigroup can be realized as a semigroup of values of a curve germ.
For a concrete example see \cite[Example 2.16]{good2}. 

(c)  \  Usually when we study properties of curve germs  starting from their semigroups, we basically prove statements for good semigroups (that is, we use properties (1)-(4), but usually not more).

(d) \ Having in mind the semigroup $\leftrightarrow$ Hilbert function correspondence from
subsection \ref{ss:2.2},  valid for curve germs, one can try to define a  `combinatorial Hilbert function' starting from any semigroup $\mathcal{S}\subset \bZ^r_{\geq 0}$ analogously to Definition \ref{def:Delta} and (\ref{eq:hfromS}). However, in the case of arbitrary $\mathcal{S}$ this procedure will not provide a well-defined $\hh$--function. On the other hand, if we start with a {\it good
semigroup} then this procedure does produce a well-defined $\hh$--function (in fact, this is equivalent to property (3) above),
which even satisfies the matroid rank inequality (\ref{eq:matroid})
(this is an exercise left to the reader). In particular,
it provides a weight function $w_0$ (via (\ref{eq:w0})), hence a lattice cohomology as well. Basically,
most of our theorems (and several theorems of the literature about curve germs)
are valid in this general context.
In this note the following results remain valid if we state them in the context of 
(the lattice cohomology associated with)  a good semigroup:
Theorem \ref{th:EUcurves},  Lemma \ref{lem:chcond},
 properties of the local minima of the weight function from section 3, the combinatorial parts
{\it (3)}$\Leftrightarrow${\it (4)}$\Leftrightarrow${\it (5)} of Theorem \ref{th:Gorchar}, the characterization of these properties via 
$\bH^*$ (as in Theorem \ref{thm:latgortechn}), or the Nonpositivity theorem 
\ref{th:Main}.  

(e) Let us exemplify statement (c): 
 from the properties (1)-(4) of a  good semigroup 
we deduce another one, which will be applied  later for semigroups of values:

\vspace{2mm}

(5)    $\Delta(\mathbf{c}-\mathbf{1})=\emptyset$.

\vspace{2mm}

This  identity for the semigroups of plane curve singularities
was proved in \cite[(2.7)]{delaMata87}. Here  we provide  a short proof for any good semigroup  based on the  `path technique' of  Lemma \ref{lem:chcond}.

Let us suppose indirectly that $\Delta(\mathbf{c}-\mathbf{1}) \neq \emptyset$, i.e. there exists an index $i \in \mathcal{I}$ and a semigroup element $s \in \Delta_i(\mathbf{c}-\mathbf{1})$. This means that $s_i= c_i-1$ and $ s_j \geq c_j$ for all $j \neq i$. By property (2)
       applied to   $s$ and $\mathbf{c}$ we get
     $ \mathbf{c} - e_i\in\mathcal{S}$.
Now, we can construct an infinite increasing path $ \mathbf{c} - e_i, \mathbf{c}, \ldots$
starting at
$\mathbf{c}-e_i$,  with all elements in $\mathcal{S}$.
Then Lemma \ref{lem:chcond} reads as $\mathbf {c}-e_i\geq \mathbf {c}$, a contradiction.
%\end{remark}

\section{Local minima of the weight function}

The semigroup of values is a rather complex invariant of curve germs; for several 
properties see e.g. \cite{delaMata87,delaMata}.
 It is hard to extract conceptual geometric information from its combinatorics. Therefore, 
it is very  desirable to introduce some other invariants, which might help to 
understand the complexity of the semigroup. 
The novelty of the present  note is that it establishes several new properties
via  the level sets of the weight function $w_0$. 
In this way,  even the older results of the literature related with the semigroup of values might 
get new deep geometric reinterpretations  in terms of the weight function and
lattice cohomology. From the technical point of view, these connections are
provided by the local minimum points of the weight function and their  $w_0$-values.
%: they  are  detectable from the $\bZ[U]$-module structure of $\mathbb{H}^{\ast}$.

%\subsection{Local minimum points} All three results of this article rely on understanding the local minima of the weight function $w$. In this subsection we introduce this notion and establish its equivalence 

\subsection{Local minimum points} \label{locmincohomology}  \cite{KNS1}
In this subsection we define the local minimum points of the weight function and collect several correspondences between them and the bigraded lattice cohomology $\bZ[U]$-module (and the respective graded root). 

Let $(C,o) = \bigcup_{i}(C_i,o)$ be a reduced  curve singularity with semigroup of values $\mathcal{S}_{C,o} \subset \bZ^r_{\geq 0}$ and weight function $w_0: \bZ^r_{\geq 0} \to \bZ$.
Sometimes we will shorten $w_0$ into $w$.

\begin{define}
    A lattice point $\elll \in \bZ_{\geq 0}^r$ is said to be a \emph{local minimum point} of the weight function $w$ if $w(\elll) < w(\elll +e_i)$ for all $i \in \mathcal{I}$
 and    $w(\elll) < w(\elll - e_i)$ for all $i \in \mathcal{I}$ with  $\elll\geq e_i$.
   % The set of local minimum points is denoted by $\mathcal{P}_{lm}$.
   All local minimum points are semigroup elements
   %i.e. $\mathcal{P}_{lm} \subseteq \mathcal{S}_{C,o}$
   (cf. (\ref{eq:Sfromh}), (\ref{eq:hfromS}) and (\ref{eq:w0valtozasa})). In particular, if $\elll$ is  a local minimum point, then  either $\elll={\mathbf 0}$ or $\elll\geq {\mathbf 1}$.

    The weight $w(\elll)$ of a local minimum point $\elll$ is said to be a \emph{local minimum value}.
    %The multiset of local minimum values is defined to be
    %$\mathcal{V}_{lm} = w(\mathcal{P}_{lm})$.
\end{define}

\begin{remark} \label{rem:locmin}
The extension of $w_0$ to higher dimensional cubes (cf. (\ref{eq:w_qdef})) assures, that $\elll$ is a local minimum point if and only if $w_0(\elll) < w_q(\square_q)$ for every $q$-cube $\square_q$ strictly containing $\elll$ as one of its vertices. Also, by (\ref{eq:w0valtozasa}), the property $w(\elll) < w(\elll \pm e_i)$ is equivalent to
%the property that 
$w(\elll) \leq w(\elll \pm e_i)$.
\end{remark}

Let us compare this notion with the local minimum points of the corresponding graded root.

\begin{define}
    $v\in\mathfrak{V} $ is a {\em local minimum point} of the
graded root $\mathfrak{R}(C, o)$ if all of its neighbours (adjacent vertices) have higher $n$-degree. By the formal definition of graded roots  (see \cite[Definition 3.2]{NOSz}) this condition is equivalent to $v$ having a single neighbour (i.e. $v$ is a \textit{leaf} of the tree $\mathfrak{R}(C, o)$), or,  $v$ having no neighbour with lower degree.
\end{define}

Notice that, by its definition, a local minimum point $\elll \in \bZ^r_{\geq 0}$ of the weight function $w$ with $w(\elll)=n$ gives a distinct connected component $\{p\}$ of $S_n$. Also, $\{p\} \cap S_{n-1} = \emptyset$, hence the vertex corresponding to this connected component is a local minimum point with degree $n$ of the graded root $\mathfrak{R}(C, o)$. By (\ref{eq:w0valtozasa}), all local minimum points of $\mathfrak{R}(C, o)$ arise in this way. Thus we have the following correspondence:
\begin{equation*}
    \left\{ \begin{array}{c}
        \text{local minimum points of } w\\
         \text{ with weight } w=n
    \end{array}
\right\}
\overset{ 1-1}{\longleftrightarrow}
\left\{ \begin{array}{c}
        \text{local minimum points of } \mathfrak{R}(C, o)\\
         \text{ with degree } n
    \end{array}
\right\}.
\end{equation*}

\begin{multicols}{2}
    Turning to the lattice cohomology $\bZ[U]$-module itself, notice that on the right hand side of this commutative diagram the direct summands corresponding to the connected components consisting of a single local minimum point span exactly the kernel of the morphism. Hence
\begin{center}
\begin{tikzcd}
	{\mathbb{H}^0_{2n}=H^0(S_n)} & {\oplus_i H^0(S_{n, i})} \\
	{\mathbb{H}^0_{2n-2}=H^0(S_{n-1})} & {\oplus_j H^0(S_{n-1,j})}
	\arrow["\cong", from=1-1, to=1-2]
	\arrow["U"', from=1-1, to=2-1]
	\arrow[from=1-2, to=2-2]
	\arrow["\cong"', from=2-1, to=2-2]
\end{tikzcd}
\end{center}
\end{multicols}
\vspace{-8mm}
\begin{equation*}
   \# \left\{ \begin{array}{c}
        \text{loc. min. pts of } w\\
         \text{ with weight } w=n
    \end{array}
\right\}
= \#
\left\{ \begin{array}{c}
        \text{loc. min. pts of } \mathfrak{R}\\
         \text{ with degree } n
    \end{array}
\right\} = \text{rank}_\bZ\ker( U : \bH^0_{2n} \to \bH^0_{2n-2} ),
\end{equation*}
i.e. we can recover the number of local minimum points with a given weight from the graded $\bZ[U]$-module structure of $\bH^{0}(C, o)$.\vspace{2mm}
%Note however that we do not have a {\it canonical} correspondence between the local minimum points and some generators or free summands of the kernel of the $U$-action.\\

The above correspondence
%we can now express the main idea of this paper: {\it The local minima of the weight function carry crucial  information about the structure of the semigroup of values and 
% create a key bridge between the semigroup and  the graded lattice cohomology module}. 
% In particular, this connection 
 manifests itself  in the three theorems
 presented in the paper in the following way. The Nonpositivity Theorem \ref{th:Main} regarding the
 weights of the reduced cohomology  follows from the fact that all local minimum points have nonpositive weight. The behaviour of the local minimum points of weight $0$ allows for the lattice cohomology to detect whether or not the  curve singularity is Gorenstein (cf. Theorem \ref{thm:latgor}). Finally, in the {\it plane curve singularity } case, the fact that the `multiplicity vector'
 has the largest weight among the  non-zero local minimum values enables one to read off the multiplicity from the lattice cohomology module (cf. Theorem \ref{th:MF}).

\begin{example}\label{ex:nagy}
 Using $\bH^*$ and $\mathfrak{R}$ let us give a pictorial presentation of the previous paragraph. 
 %for $\bH^0$, we can now summarize the main results of this paper in the image below.

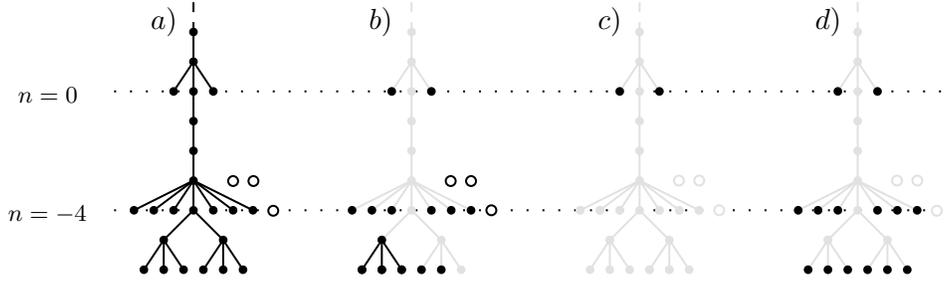
\begin{center}
\tikzset{every picture/.style={line width=0.75pt}} %set default line width to 0.75pt        


\end{center}

 Diagram \textit{a)} depicts the graded root and $\bH^1$ of the plane curve singularity $(C,o)$ given by the equation $(x^2-y^7)(x^5-y^4)=0$. Here, the dots represent the vertices of $\mathfrak{R}$, i.e. $\bZ$-summands of $\bH^0$ and the circles represent $\bZ$-summands of $\bH^1$ (each of them placed on the appropriate  weight level).

The black dots and circles in diagram \textit{b)} correspond to $\bH^*_\text{red}$
and the edges to the $U$--action
(by a non-canonical isomorphism). 
%Note that while $\bH^*_\text{red}$ is well defined as a $\bZ[U]$-module, there is no canonical way to represent it as a graded root, whence the edges of $\mathfrak{R}$ have also been lightened.
Observe that the set of weights  $n$ such that $ \bH^*_{\text{red}, \, 2n}(C,o)\not=0 $ is indeed bounded from above by $0$ (corresponding to the Nonpositivity Theorem \ref{thm:intr-nonpos}).

Diagram \textit{c)} shows the local minima of weight $n=0$. The fact that there are two of them 
implies that $(C,o)$ is Gorenstein (see Theorem \ref{thm:mult-intr}\textit{(a)}).

Finally, diagram \textit{d)} depicts all of the local minima. Among the ones with weight $n<0$ the maximal weight is $n=-4$, the weight of $\mathbf{m}$.  This corresponds to the fact that the multiplicity of $(C,o)$ is $2-(-4)=6$ (see Theorem \ref{thm:mult-intr}\textit{(b)}).
\end{example}

%Before we start the technical discussion, we wish to stress another observation. Usually, symmetry properties of invariants (in the present case, of the semigroup of values, or of the weight function) are related with the presence of some Gorenstein (or other type of) dualities. However, below we will face  an interesting and surprising phenomenon, namely  the presence of some kind of (`local', `weak') symmetry in certain rectangles  associated with generalized local minima,  even if we do not impose a priori any Gorenstein property of the singular germ.
%Several  properties listed  below should be interpreted  from this  point of view.

\subsection{Generalized local minima} It turns out that in order to understand the local minima of the weight function, it is convenient to use a slightly more general notion, namely that of \emph{generalized local minima}. In this subsection we give a detailed discussion of this notion. It forms the technical backbone of this paper and will be used in the proofs of all three main results.

\begin{define}
    A lattice point $\elll \in \bZ^r_{\geq 0}$ satisfying the condition
      $  w(\elll)< w(\elll-e_i)$ for any $i \in \mathcal{I}$ with $\elll\geq e_i$
     is said to be a \emph{generalized local minimum point} of the weight function.
\end{define}
% By definition, any local minimum point is a generalized local minimum point.
 %In particular,
 %$ {\mathbf 0}$ is a generalized local minimum point.
\begin{lemma} \label{lem:lm}
    Let $\elll>{\mathbf 0}$ be a nonzero lattice point.
    %Set $\mathbf {1}_\elll:= \sum_{j|\elll\geq e_j} e_j)$.
    Then   (for notation see Definition \ref{def:Delta})

 (a)   $\elll$ is a generalized local minimum point
  if and only if  $\Delta(\elll - \mathbf {1}) = \emptyset$;

 (b)        $\elll$ is a  local minimum point  if and only if
       $\elll \in \mathcal{S}_{C,o}$ and $\Delta(\elll - \mathbf{1}) = \emptyset$.
\end{lemma}

\begin{proof} \textit{(a)}
 By (\ref{eq:w0valtozasa}), the property  $w(\elll-e_i) > w(\elll)$, for a certain  $i \in \mathcal{I}$, %(when $\elll\geq e_i$)
 is equivalent to $\overline{\Delta}_i(\elll-e_i) = \emptyset$. But $\overline{\Delta}_i(\elll-e_i) = \Delta_i(\elll- \mathbf{1})$
 and
% (since a nonzero semigroup has all its coordinates positive).
%  Hence $\elll>{\mathbf 0}$ is a generalized local minimum if and only if $
%\bigcup_{i|\elll\geq e_i} \Delta_i(\elll - \mathbf{1}_\elll) = \emptyset$.
%Since   $\Delta_j(\elll - \mathbf{1}_\elll)=
 % \emptyset$ for any $j$ with $\elll\not\geq e_j$, the second condition is equivalent with
   $ \Delta(\elll - \mathbf {1})= \bigcup_{i\in \mathcal {I}} \Delta_i(\elll - \mathbf{1})$.

\textit{(b)} Use  \textit{(a)} and the fact that
 $w(\elll+e_i) > w(\elll)$ for all $i $ is equivalent to $\elll \in \mathcal{S}_{C,o}$, cf. (\ref{eq:Sfromh}).
\end{proof}

\begin{example}\label{ex:c} 
%(a)
%For $r=1$, 
%$p\geq 1$ is a generalized local minimum point if and only if $p-1\not \in \calS_{C,o}$.
%
%(b)
%For $r\geq 2$, Delgado in \cite{delaMata87,delaMata}
%called  a semigroup element $\alpha\in \calS_{C,o}$  
%{\it maximal} if $\Delta(\alpha)=\emptyset$.
%If this happens, then $\elll:= \alpha+\mathbf{1}$ is a generalized local minimum point
%of the weight function
%(with the additional properties $\elll\geq \mathbf{1}$ and $\elll-\mathbf {1} 
%\in\calS_{C,o}$). On the other hand, if $\elll\geq \mathbf{1}$ is a 
%generalized local minimum point
%then $\elll-\mathbf{1}$ may or may not be a semigroup element.
%(if it is, then $\elll-\mathbf{1}$ is `maximal').
There exist three `canonical' local minimum points, namely $\mathbf{0}$,
$\mathbf {m}$ and
$\mathbf{c}$.

For $r=1$ the fact that  $\mathbf{c}$ is a local minimum point follows directly from the definition.  For $r\geq 2$ use (5) from subsection \ref{rem:semigroup}.
%(In the Gorenstein case $\mathbf{c}-\mathbf {1}$ is a maximal semigroup element, 
%cf. Theorem \ref{th:Gorchar} or  \cite[(2.7)]{delaMata87}, 
%but in general $\mathbf{c}-\mathbf {1}\not\in\mathcal{S}$,
%see e.g. the decomposable germs.)

 If  $\mathbf{0}$, $\mathbf {m}$ and $\mathbf{c}$ are all distinct then
the number  of the minimal $\bZ[U]$-module generators of
%$\bH^0$ has at least three $\bZ[U]$-module generator (cf. Remark \ref{locmincohomology}), hence
$\bH^0_{{\rm red}}$ is at least two. On the other hand,
$\mathbf {m}=\mathbf{c}$ is equivalent with ${\rm rank}_{\bZ}\bH^0_{{\rm red}}=1$
(and also with ${\rm rank}_{\bZ}\bH^*_{{\rm red}}=1$). (Compare with Remark \ref{rem:0locmin}.)

\end{example}

The very existence of the conductor $\mathbf{c}$
of the semigroup of values guarantees that there are only finitely many generalized local minimum points.   Indeed, Lemma \ref{lem:lm} implies the following.

\begin{corollary}\label{cor:kisebb}
If  $\elll$ is a generalized local minimum point then $\mathbf{0}\leq \elll\leq \mathbf{c}$.
\end{corollary}

\begin{lemma} \label{gyengeGor}
    Let $\elll \in \bZ_{\geq 0}^r$ be a generalized local minimum point. Then the two statements
    $$\ell \in \mathcal{S}_{C, o} \ \ \ \ \text{ and } \ \ \ \ \Delta(\elll-\mathbf{1}-\ell) \not= \emptyset $$
    cannot be true at the same time for any lattice point $\ell \in \bZ^r$.
\end{lemma}

\begin{proof}
    Assume indirectly that $\ell \in \mathcal{S}_{C,o}$ and that there exists an element $\ell' \in
    \Delta(\elll-\mathbf{1}-\ell)$. Then $\ell + \ell' \in \Delta(\elll-\mathbf{1})$, a fact which contradicts  Lemma \ref{lem:lm}{\it (a)}.
\end{proof}

\begin{corollary} \label{szimmutak}
    Let $\elll \in \bZ_{\geq 0}^r$ be a generalized local minimum point and $i\in\mathcal{I}$. Then
    \begin{equation}\label{eq:ujszam}
    w(\ell) < w(\ell + e_i) \ \ \text{ implies  } \ \ w(\elll-e_i-\ell) > w(\elll-\ell)
    \end{equation}
    for any lattice point $\ell\in \bZ_{\geq 0}^r$ satisfying  $\ell+e_i\leq \elll$.
    %$\ell \in \bZ_{\geq 0}^r$.\in \bZ_{\geq 0}^r
\end{corollary}

\begin{proof}
    The inequality $w(\ell) < w(\ell+e_i)$ is equivalent to the existence of an element  $\ell' \in \overline{\Delta}_i(\ell)$, see (\ref{eq:w0valtozasa}). Since $\ell'\in\mathcal{S}_{C,o}$,
    by Lemma \ref{gyengeGor} it follows that
    $\Delta(\elll-\mathbf{1}-\ell') = \emptyset$. But $\Delta(\elll-\mathbf{1}-\ell') \supseteq \Delta_i(\elll-\mathbf{1}-\ell') \supseteq \overline{\Delta}_i(\elll-e_i-\ell)$, so $\overline{\Delta}_i(\elll-e_i-\ell) = \emptyset$, which implies
    $w(\elll-e_i-\ell) > w(\elll-\ell)$.
\end{proof}

\begin{proposition}\label{cl:glm}
Let $\elll \in \bZ^r_{\geq 0}$ be a generalized local minimum and let
$\gamma = \{x^k\}_{k=0}^t$ be an increasing path from $\mathbf{0}$ to  $\elll$. Then
     \begin{equation}\label{eq:w(glm)}
         w(\elll) = - \#\{x^k\,\vert \, w(x^{k+1})=w(x^k)-1 \text{ and }  w(\elll-x^{k})=w(\elll-x^{k+1})-1 \text{ with } 0 \leq k < t\}.
     \end{equation}
In particular, the right hand side is independent of the increasing path and $w(\elll) \leq 0$ for any  generalized local minimum $\elll$.
\end{proposition}

\begin{proof}
%By Lemma \ref{lem:wrestricted} we may assume that $c \geq \mathbf{1}$. Indeed, otherwise we could use the sublattice corresponding to the non-zero components of $c$.
Let us choose an increasing path $\gamma = \{x^k\}_{k=0}^t$ from $\mathbf{0}$ to $\elll$ and define the dual path $\overline{\gamma}=\{ \overline{x}^{k}\}_{k=0}^t$ given by $\overline{x}^k := \elll-x^{t-k}$. Clearly, this is also an increasing path from $\overline{x}^0=\elll - x^t= \mathbf{0}$ to $\overline{x}^t=\elll - x^0=\elll$. By telescoping summation we have
    \begin{align*}
        w(\elll) = w(\elll)-w(\mathbf{0}) &= \frac{1}{2}\bigg( \sum_{k=0}^{t-1}\big( w(x^{k+1})-w(x^k)\big) + \sum_{k=0}^{t-1}\big( w(\overline{x}^{t-k})-w(\overline{x}^{t-k-1})\big)\bigg).
    \end{align*}
    By Lemma \ref{szimmutak}, an increase in $w$ along one of the two paths implies a decrease in $w$ along the other one, whence the total sum in the right hand side (hence  $w(\elll)$ too) cannot be positive, 
     and it equals $(-1)$-times the number of symmetric pairs of edges $((x^k, x^{k+1}), (\overline{x}^{t-k-1}, \overline{x}^{t-k}))$ with decreasing $w$ along both.
    \end{proof}

\begin{theorem}\label{lem:w=0glm}
(a)    If $\elll \in \bZ_{\geq 0}^r$ is a generalized local minimum point with weight $w(\elll)=0$, then
    \begin{equation}\label{eq:sym}
        w(\elll-\ell) = w(\ell) \hbox{ for any lattice point } \ell \in R(\mathbf{0}, \elll) \cap \bZ^r,
    \end{equation}
    i.e.
    the weight function restricted to the rectangle $R(\mathbf{0},\elll)$ is symmetric. Furthermore if $p \geq \mm$ then $p$ is in fact a local minimum point (i.e. $p \in \mathcal{S}_{C,o}$).

 (b)    Assume that  $m(C,o) \geq 3$.
 If $\elll$ is a local minimum point with weight $w(\elll)=0$, then  $\elll\in\{ \mathbf{0},
 \mathbf{c}\}$.
\end{theorem}

\begin{proof} \textit{(a)} First we verify that
 the converse implication of the statement (\ref{eq:ujszam})  also holds.

Consider an  increasing path $\gamma=\{ x^k\}_{k=0}^t \subset R(\mathbf{0}, \elll)$ with $x^0=\mathbf{0}$ and $x^t=\elll$. Then  using  (\ref{eq:w(glm)})  and the vanishing $w(\elll)=0$
     we obtain  the following identity for any $0 \leq k < t$:
\begin{equation*}
    w(\elll-x^{k})-w(\elll-x^{k+1})=-\big(w(x^{k+1})-w(x^k)\big).
\end{equation*}
However, for any pair of lattice points $\ell, \ell+e_i \in R(\mathbf{0}, \elll) \cap \mathbb{Z}^r$ we can choose an increasing path $\gamma$ from $\mathbf{0}$ to $\elll$
containing both $\ell$ and $\ell + e_i$, and thus
\begin{equation}\label{eq:dwsymhaGor}
    w(\elll-\ell-e_i)-w(\elll-\ell)=w(\ell+e_i)-w(\ell) \ \hbox{ for any } \ \ell, \ell + e_i \in R(\mathbf{0}, \elll) \cap \mathbb{Z}^r.
\end{equation}

Next, we prove the symmetry (\ref{eq:sym})
by induction on  $|\ell|=\sum_i \ell_i$ for the lattice points
$\ell \in R(\mathbf{0}, \elll) \cap \mathbb{Z}^r$. The starting case of the induction
 ($|\ell|=0$)
is just the assumption: $w(\elll)=0= w(\mathbf{0})$.

Suppose we have $w(\elll - \ell)= w(\ell)$ for every lattice point $\ell \in R(\mathbf{0}, \elll) \cap \bZ^r$ with $|\ell|\leq j$. Now, any lattice point $\ell'\in R(\mathbf{0}, \elll) \cap \bZ^r$ with $|\ell'|= j+1$ can be written as $\ell + e_i$ for some $\ell \in R(\mathbf{0}, \elll) \cap \bZ^r$ with $|\ell|= j$. Then $w(\elll - \ell)= w(\ell) $ and by (\ref{eq:dwsymhaGor}) we get
$w(\elll-\ell-e_i)=w(\ell+e_i)$.

Finally, we show that $p \in \mathcal{S}_{C,o}$ if $p \geq \mm$. Using the symmetry of $w$ established above, $\mm$ being a local minimum point implies that $p-\mm$ is a local minimum point as well. In particular, $p-\mm \in \mathcal{S}_{C,o}$, whence $p = (p-\mm) +\mm \in \mathcal{S}_{C,o}$.

\vspace{2mm}

%Hence, by induction, the weight function is symmetric.
%\end{proof}

%\begin{prop} \label{prop:locminw0cond}
%\end{prop}

%\begin{proof}
 \textit{(b)} Assume that $\elll\geq \mm$.
 %By Lemma \ref{lem:lm}, $\elll \in \mathcal{S}_{C, o}$ and $\Delta(\elll-\mathbf{1}) = \emptyset$. Therefore
 Since $\elll$ is a generalized local minimum point with weight $w(\elll) = 0$,
 by part \textit{(a)} of this Theorem \ref{lem:w=0glm}, the weight function is symmetric in the rectangle $R(\mathbf{0}, \elll)\cap \mathbb{Z}^r$, i.e. $w(\ell) = w(\elll-\ell)$ for all $\ell \in R(\mathbf{0},\elll) \cap \bZ^r$.

    As $\mathcal{S}_{C,o}\cap \big(R(\mathbf{0},\elll) \setminus R(\mm,\elll)\big) = \{\mathbf{0}\}$ (see subsection \ref{ss:multlocmin}), we have that $w(\mathbf{0}) = 0$ and $w(\ell) = 2-|\ell|$ for any $\ell \in R(\mathbf{0},\mm)\cap \bZ^r \setminus \{ \mathbf{0} \}$. Now, using the symmetry of $w$ 
    on $R(\mathbf{0},\elll)$ 
    and the fact that $p\in \mathcal{S}_{C,o}$, then via (\ref{eq:Sfromh}) and property (2) of subsection \ref{rem:semigroup} we get
    $$\mathcal{S}_{C,o}\cap R(\elll-\mm,\elll) = R(\elll-\mm,\elll)\cap\bZ^r
    \setminus \{\elll-e_i, \dots, \elll-e_r \}.$$
    By adding the multiplicity vector $\mm\in \mathcal{S}_{C,o}$ to every element of this set, we get that
    $$R(\elll, \elll+\mm)\cap\bZ^r \setminus \{ \elll+\mm-e_1, \dots, \elll+\mm-e_r\} \subseteq \mathcal{S}_{C,o}.$$

    Notice that $\elll+\mm-e_i \in \mathcal{S}_{C,o}$ for at least one $i \in \mathcal{I}$: if this was not the case, then by property (2) of subsection \ref{rem:semigroup} and Lemma \ref{lem:lm}, $\elll + \mm$ would be a local minimum point with $w(\elll+\mm) = m(C,o)-2 \geq 1$ contradicting Proposition \ref{cl:glm}. (This is where we use the assumption $m(C,o) \geq 3$.)
    This `new' semigroup element yields an increasing path  $\gamma=\{ x^k\}_{k=0}^m$ from $\elll$ to $\elll + \mm$ with $x^{m-1} = p+\mm-e_i$ consisting of semigroup elements only. It follows that
    $\bigcup_{k =0}^\infty (k\cdot \mm + \gamma )$
    is an infinite increasing path starting at the point $p$ and consisting of only semigroup elements. 
     Then  $\elll \geq \mathbf{c}$ by Lemma \ref{lem:chcond}. But by Corollary \ref{cor:kisebb}
     $\elll \leq \mathbf{c}$ too.
\end{proof}

\begin{example}\label{ex:m=2}
 If $m(C,o) =2$, i.e. $(C,o)$ is the plane curve singularity $x^2+y^k$ ($k\geq 2$), then there can be more local minimum points with weight $w(\elll)=0$ (in fact, they are multiples of $\mm$).
\end{example}

\section{First application. The Gorenstein property}

\subsection{Gorenstein symmetry of the weight function}\label{ss:GORduality} In this subsection we prove two  characterizations for Gorenstein curve singularities. 
The first one, Theorem \ref{th:Gorchar},  shows the relevance of the weight function $w$ and its strong connection with the semigroup of values and the Hilbert function. The equivalences  \textit{(1)} $\Leftrightarrow$ \textit{(2)} $\Leftrightarrow$ \textit{(5)} are known (see references in the proof), parts
 \textit{(3)}--\textit{(4)}  are new. They formulate the Gorenstein  property 
 in a natural and useful way in terms of $w$.

\begin{theorem}\label{th:Gorchar}
    Let $(C, o)$ be a curve singularity. Then the following conditions are equivalent:
    \begin{enumerate}
        \item $(C, o)$ is Gorenstein, i.e. the local ring $\mathcal{O}_{C,o}$ is Gorenstein;
        \item $\dim_{\mathbb{C}} (\overline{\mathcal{O}}/\mathcal{O})=\dim_{\mathbb{C}}(\mathcal{O}/\mathfrak{c})$;
        \item $w(\mathbf{c}) = w(\mathbf{0})=0$;
        \item $w$ is `Gorenstein symmetric', i.e. $w(\mathbf{c}- \ell) = w(\ell)$ for any lattice point $\ell \in R(\mathbf{0}, \mathbf{c}) \cap \mathbb{Z}^r$;
        \item the semigroup of values $\mathcal{S}_{C,o}$ is symmetric (in the sense of Garcia-Delgado), i.e.
        \begin{equation*}
            \ell\in\mathcal{S}_{C,o}\ \Leftrightarrow \ \Delta(\mathbf{c}-{\bf 1}-\ell)=
            \emptyset.
        \end{equation*}
        Specifically, for irreducible curves the semigroup $\mathcal{S}_{C,o}$ is symmetric if and only if
        \begin{equation*}
            \ell\in\mathcal{S}_{C,o}\ \Leftrightarrow \ \mathbf{c}- 1-\ell\notin \mathcal{S}_{C,o}.
        \end{equation*}
    \end{enumerate}
In particular, if $(C,o)$ is Gorenstein, then by {\it (4)} the involution
$\ell\leftrightarrow \mathbf{c}-\ell$ induces a $\bZ_2$--symmetry of the graded root and of the lattice cohomology $\bH^*$  as well.
\end{theorem}

\begin{proof}
By Serre \cite{Serre} or Bass  \cite{Bass} (see also \cite{Huneke}) $(C,o)$ is Gorenstein if and only if
$\dim (\overline{\mathcal{O}}/\mathcal{O})=\dim(\mathcal{O}/\mathfrak{c})$, i.e. \textit{(1)} $\Leftrightarrow$ \textit{(2)}. Also, \textit{(2)} $\Leftrightarrow$ \textit{(3)} follows  as (see also subsection \ref{ss:cond})
\begin{align*}
    \dim (\overline{\mathcal{O}}/\mathcal{O})=\dim(\mathcal{O}/\mathfrak{c}) & \Leftrightarrow  \dim (\overline{\mathcal{O}}/\mathfrak{c}) = 2  \dim(\mathcal{O}/\mathfrak{c}) \\
    & \Leftrightarrow |\mathbf{c} | = 2  \mathfrak{h}(\mathbf{c}) \\
    & \Leftrightarrow w(\mathbf{c})=0. \label{eq:w(c)=0haGor}
\end{align*}

The equivalence \textit{(2)} $\Leftrightarrow$ \textit{(5)} was proved by Kunz ($r=1$)
\cite{Kunz},
Garcia and Waldi \cite{waldi,Garcia} ($r=2$) and Delgado ($r\geq 2$) \cite{delaMata}.
See also \cite{cdk}.

The \textit{(4)} $\Rightarrow$ \textit{(3)} implication is trivial,
 \textit{(3)} $\Rightarrow$ \textit{(4)}
is just  Theorem \ref{lem:w=0glm}\textit{(a)} (via Example  \ref{ex:c}).

\vspace{2mm}

 This basically ends the proof of the theorem. However, 
we give below two direct and short proofs for the \textit{(4)} $\Rightarrow$ \textit{(5)}
and \textit{(5)} $\Rightarrow$ \textit{(3)} implications based on
 the tool of increasing paths and the
correspondence between the  weight function $w$ and the semigroup $\mathcal{S}_{C,o}$.
(In this way we provide a self-contained proof for
\textit{(2)} $\Leftrightarrow$ \textit{(3)} $\Leftrightarrow$
\textit{(4)} $\Leftrightarrow$ \textit{(5)} replacing the proofs from
\cite{Kunz,waldi,Garcia,delaMata}.
The implication \textit{(1)} $\Rightarrow$ \textit{(4)} was  proved in \cite{AgNeIV} as well.)

\vspace{2mm}

\textit{(4)} $\Rightarrow$ \textit{(5)}
    By (\ref{eq:Sfromh}) and (\ref{eq:w0valtozasa}) $\ell$ is a semigroup element if and only if $w(\ell + e_i) > w(\ell)$ for every $i \in \mathcal{I}$. But, by the $\ell \leftrightarrow \mathbf{c}- \ell$ symmetry of the weight function, this is equivalent to the fact that $\mathbf{c}- \ell$ is a generalized local minimum point, i.e. by Lemma \ref{lem:lm}\textit{(a)}
    $\Delta(\mathbf{c} - \mathbf{1} -\ell) = \emptyset$.

  \textit{(5)} $\Rightarrow$ \textit{(3)}. For any $\ell \in R(\mathbf{0}, \mathbf{c}) \cap \mathbb{Z}^r$ set $D w(\ell):= w(\ell)-w(\mathbf{c}-\ell)$.
  First, we construct  an increasing path $\{x^k\}_k$ connecting $\mathbf{0}$ with $\mathbf{c}$
  such that $\ell\mapsto Dw(\ell)$ is constant along this path.

  Assume that $x^k< \mathbf{c}$ was already constructed, and we wish to find $x^{k+1}=x^k+e_{i(k)}\leq \mathbf{c}$.

Set $J:=\{ j\in \mathcal{I}\,\vert \, x^k_j < c_j\}$. Note that $J\not=\emptyset$.
  Assume that there exists a certain $j\in J$ such that for
   $x^{k+1}=x^k+e_{j}$ one has
  $w( x^{k+1})> w( x^{k})$. Then, by Corollary \ref{szimmutak}
  %, there exists a semigroup element $
  %s\in   \overline{\Delta}_j(x^k)$. By  \textit{(5)} $\Delta (\mathbf{c}-\mathbf{1}-s)=
  %\emptyset$. Hence $ \emptyset = \Delta (\mathbf{c}-\mathbf{1}-s)\supset
  %\overline{\Delta}_j (\mathbf{c}-e_j-s) \supset
  %\overline{\Delta}_j (\mathbf{c}-x^{k+1})$, hence
 we also have  $w(\mathbf{c}- x^{k})<
 w(\mathbf{c}- x^{k+1})$.
 Hence with the choice   $x^{k+1}=x^k+e_{j}$ we have $x^{k+1}\leq \mathbf{c}$ and
$ D w(x^{k+1})= D w(x^{k})$ (see \ref{eq:w0valtozasa}).

Now assume the contrary, i.e.  $w( x^{k}+e_j)< w( x^{k})$ for every $j\in J$. Then $ x^{k}\not\in \calS_{C,o}$.
 Hence, by  \textit{(5)},
 %$\Delta (\mathfrak{c}-\mathbf{1}-x^{k})\not = \emptin yset$. This means that
 there exists some $i$  such that
 $(\dag)$ $\Delta_i (\mathbf{c}-\mathbf{1}-x^{k})\not = \emptyset$. Note that for any
  $l\not\in J$ one has  $\Delta_l (\mathbf{c}-\mathbf{1}-x^{k}) = \emptyset$.
 Hence,  such an $i$ with $(\dag)$ should be in $J$.
  But this implies
  that  $\overline{\Delta}_i (\mathbf{c}-e_i-x^{k})\not = \emptyset$, hence
  $w(\mathbf{c}- e_i-x^{k})<
 w(\mathbf{c}- x^{k})$. For such an $i$ set  $x^{k+1}=x^k+e_{i}$, hence again
 $x^{k+1}\leq \mathbf{c}$ and
 $ D w(x^{k+1})= D w(x^{k})$ by Corollary \ref{szimmutak}.

 Then, by the existence of such a path, $ D w(\mathbf{0})= D w(\mathbf{c})$, i.e. $ -w(\mathbf{c})=
  w(\mathbf{c})$, hence  $w(\mathbf{c})=0$.
%    \textit{(5)} $\Rightarrow$ \textit{(3)}. Suppose indirectly that the semigroup of values $\mathcal{S}_{C,o}$ is not symmetric. Claim \ref{clm:Delta(c-1)} and Lemma \ref{gyengeGor} imply that we cannot have $\ell \in \mathcal{S}_{C,o}$ and $\Delta(\mathbf{c}-\mathbf{1} - \ell) \neq \emptyset$ simultaneously, so the indirect assumption is that there exists some lattice point $\ell \notin \mathcal{S}_{C,o}$ with $\Delta(\mathbf{c}-\mathbf{1} - \ell) = \emptyset$. Then (\ref{eq:Sfromh}) and (\ref{eq:w0valtozasa}) imply that there exists a coordinate direction $i \in \mathcal{I}$, such that $w(\ell + e_i) = w(\ell)-1$. On the other hand, by Lemma \ref{lem:glm} the lattice point $\mathbf{c}- \ell$ is a generalized local minimum point, hence it has $w(\mathbf{c}-\ell) < w(\mathbf{c}-\ell -e_i)$. But then, if we choose a strictly increasing path $\gamma=\{x^k\}_{k=0}^t$ from $\mathbf{0}$ to $\mathbf{c}$ going through the lattice points $\ell$ and $\ell+e_i$, Corollary \ref{cl:glm} implies that $w(\mathbf{c})<0$, contradiction to condition \textit{(3)}.
\end{proof}

\begin{example}\label{ex:non}
It is a slightly surprising fact that the Gorenstein property is non-hereditary. That is, from the
fact that $(C,o)=\bigcup_{i\in\mathcal{I}} (C_i,o)$ is Gorenstein, it does not necessarily  follow that
any $(C_J, o)=\bigcup_{i\in J} (C_i,o)$ ($\emptyset\not=J\subset \mathcal {I}$) is Gorenstein.

Assume for example that $(C,o)$ is the union of four generic lines in $(\mathbb{C}^3,0)$. It can be
realized as the zero set of two homogeneous quadratic equations. In particular,
 it is a complete intersection, hence Gorenstein. On the other hand, the union of any three generic lines $(C',o)$  in $(\mathbb{C}^3,0)$ is $\vee_{i=1}^3 (\bC,0)$, which  is
 not Gorenstein.
  E.g., the conductor  is $(1,1,1)$ with $w(1,1,1)=-1$.
 (The reader  is invited to work out and compare the semigroups and  $w$--tables of $(C,o)$ and $(C',o)$.)

%For similar examples with $r=2$ see ???????????????????????.

\end{example}
%\begin{example}\label{ex:T}
%Let us consider the semigroup $\mathcal{S}$ of an irreducible  plane curve singularity. %Clearly,
%it is Gorenstein symmetric, hence any analytic realization
%$(C,o)$ with $\mathcal{S}_{C,o}=\mathcal{S}$ is Gorenstein.
%The point is that (by a result of Teissier), in fact, any such $(C,o)$ is a complete %intersection \cite[Appendix]{Zar}.
%
%\end{example}

\subsection{Gorenstein property and lattice cohomology} Next we prove that the graded $\bZ[U]$-module $\bH^0(C,o)$ of a reduced curve singularity $(C,o)$ determines whether
$(C,o)$ is Gorenstein.

\begin{theorem}\label{thm:latgor}
    The reduced  curve singularity $(C,o)$ is Gorenstein if and only if
    $$\textit{either} \ \ \ \bH^0_{\text{\normalfont red}}(C,o) = 0 \ \  \ \textit{or} \ \ \
    \text{\normalfont rank}_\bZ \ker\big( U : \bH^0_0(C,o) \to \bH^0_{-2}(C,o) \big) \geq 2.$$
\end{theorem}

The more detailed version of this theorem states the following.

\begin{theorem}\label{thm:latgortechn}
Let $(C, o)$ be a reduced  curve singularity. Then we have the following facts.
\begin{itemize}
    \item[(a)] If $\bH^0_{\text{\normalfont red}}(C,o) = 0$, then $(C,o)$ is smooth, hence Gorenstein.
    \item[(b)] If $\bH^0_{\text{\normalfont red}}(C,o) \not= 0$ but $\bH^0_{2n}(C,o) = 0$ for all $n < 0$, then $m(C,o)=2$, $(C,o)$ is a plane curve singularity
        (hence Gorenstein)
      and  $\text{\normalfont rank}_\bZ \ker\big( U : \bH^0_0(C,o) \to \bH^0_{-2}(C,o) \big) \geq 2$.
         % and hence $(C,o)$ is Gorenstein.
    \item[(c)] If $\bH^0_{2n}(C,o) \not= 0$ for some $n<0$ and $\text{\normalfont rank}_\bZ \ker\big( U : \bH^0_0(C,o) \to \bH^0_{-2}(C,o) \big) = 2$, then $(C,o)$ is Gorenstein.
    \item[(d)] In any other case $\text{\normalfont rank}_\bZ \ker\big( U : \bH^0_0(C,o) \to \bH^0_{-2}(C,o) \big) = 1$ and $(C,o)$ is not Gorenstein.
\end{itemize}

In cases {\rm (b)} and {\rm (c)} $\ker\big( U : \bH^0_0(C,o) \to \bH^0_{-2}(C,o) \big)$ contains the free
$\mathbb{Z}$-submodule of rank 2 generated by connected components of $S_0$ corresponding to the local minimum points $\mathbf{0}$ and $\mathbf{c}$.
\end{theorem}

\begin{proof}
For (a) and (b) use subsection
\ref{ss:multlocmin} and Proposition \ref{thm:nonnegw}. Otherwise, the multiplicity is $>2$ and
$\bH^0_{<0}(C,o) \not= 0$ (see again Proposition \ref{thm:nonnegw}).

If $\text{\normalfont rank}_\bZ \ker\big( U : \bH^0_0(C,o) \to \bH^0_{-2}(C,o) \big) \geq 2$, then there exists at least one non-zero local minimum point $\elll \in \bZ^r_{\geq 0}$ with $w(\elll)=0$
 (cf. subsection \ref{locmincohomology}). 
 But by Theorem \ref{lem:w=0glm}{\it (b)} any such $\elll$ must be the conductor, whence $\text{\normalfont rank}_\bZ \ker\big( U : \bH^0_0(C,o) \to \bH^0_{-2}(C,o) \big) = 2$
        (generated by $\{\mathbf{0},\mathbf{c}\}$)  and $(C,o)$ is Gorenstein by Theorem \ref{th:Gorchar}{\it (3)}.

        In any other case $\text{\normalfont rank}_\bZ \ker\big( U : \bH^0_0(C,o) \to \bH^0_{-2}(C,o) \big) = 1$, so $\mathbf{0} \in \bZ^r_{\geq 0}$ is the only local minimum with $w=0$ (see Remark \ref{rem:0locmin}), whence $w(\mathbf {c}) < 0$ and  $(C,o)$ cannot be Gorenstein again by Theorem \ref{th:Gorchar}{\it (3)}.
\end{proof}

\section{Second application. Multiplicity and lattice cohomology}\label{multlattice}

\subsection{The Multiplicity Formula}\label{ss:5.1}
Let $(C, o)$ be a reduced  curve singularity.
Our goal is to identify  $w(\mathbf {m})=2-m(C,o)$ from the graded $\bZ[U]$--module $\bH^0(C,o)$.

By  Theorem \ref{th:Main},  $\bH^*_{{\rm red}}$ is supported in nonpositive weights, i.e.   $\bH^*_{\text{red},> 0}=0$.

\begin{define} For any $(C,o)$ with $0^{th}$ lattice cohomology
$\bH^0=\bH^0(C,o)$ we define
$$
M(\bH^0)=\left \{\begin{array}{ll}
1 & \mbox{if $\bH^0_{\text{red}}=0$},\\
0 & \mbox{if $\bH^0_{\text{red}}\not =0$, $\bH^0_{<0}=0$},\\
\max\{w\,|\,  {\rm ker} (U:\bH^0_{2w}\to \bH^0_{2w-2}) \not =0,\ w<0\}&
 \mbox{if  $\bH^0_{<0}\not =0$}.\end{array}\right. $$
 We say that $(C,o)$ satisfies the `Multiplicity Formula' (in short `MF')  if
\begin{equation}\tag{MF}\label{eq:MF}
w(\mathbf {m})=M(\bH^0), \ \mbox{i.e.} \   m(C, o)=2-M(\bH^0).
\end{equation}
\end{define}

If   $\bH^0_{<0}=0$ then the  identity  $2-m(C,o)=M(\bH^0)$ holds by Proposition \ref{thm:nonnegw}. If   $\bH^0_{<0}\not =0$ (i.e. when $m(C,o)\geq 3$) any local minimum $\elll$ satisfies
$w(\elll)\leq 0$ (Proposition \ref{cl:glm}), and  if $w(\elll)=0$ then $\elll\in\{\mathbf {0},\ccc\}$ (cf. Theorem
 \ref{lem:w=0glm}\textit{(b)}).  Hence, for $m(C,o)\geq 3$,  via subsection \ref{locmincohomology},   MF is equivalent to
%the following fact:
%for any local minimum $\elll\not \in\{\mathbf {0},\ccc\} $ (hence  with $w(\elll)<0$) we have
%$w(\elll)\leq w(\mathbf{m})$. In other words, if  $m(C,o)\geq 3$ then
%MF reads as
 $$w(\mathbf {m})=\max\{ w(\elll)\,\vert \, \mbox{$\elll$ local minimum with $w(\elll)<0$}\}
=\max\{ w(\elll)\,\vert \, \mbox{$\elll$ local minimum $\not\in\{\mathbf{0},\mathbf{c}\}$}\}
.$$
Since any local minimum is a semigroup element, for any testing $\elll$ we can assume that
$\elll\geq  \mathbf {m}$, or,  in the Gorenstein symmetric case, that $\mathbf {m}\leq \elll\leq \ccc-\mathbf {m}$.

\begin{example}\label{ex:45}

Consider the following four subsemigroups  of $\bZ_{\geq 0}$. In all these cases
$\bH^{>0}=0$ (cf. Remark \ref{rem:lhom}) hence all the lattice cohomological information is coded in the corresponding graded roots. They are listed below, where we also
note the validity of MF.

%\begin{center}

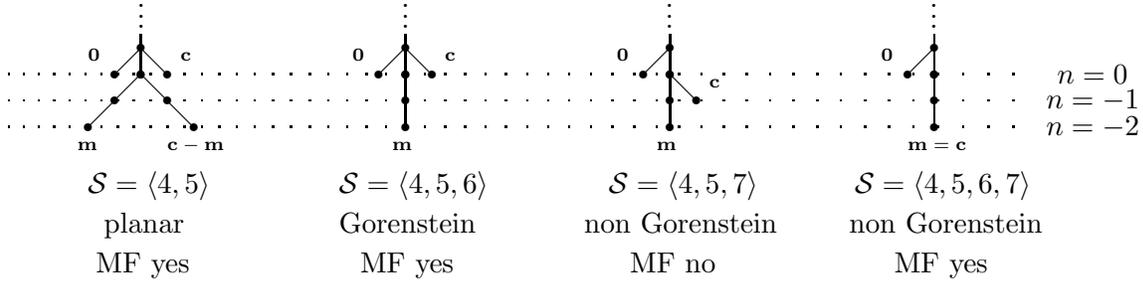
\begin{picture}(400,120)(50,-40)
\linethickness{.5pt}

\put(100,40){\circle*{3}}
\put(100,50){\circle*{3}}
%\put(100,30){\circle*{3}}
\put(100,40){\circle*{3}}
%\put(100,10){\circle*{3}}
\put(90,40){\circle*{3}}
\put(110,40){\circle*{3}}
\put(90,30){\circle*{3}}
\put(80,20){\circle*{3}}
\put(120,20){\circle*{3}}
\put(110,30){\circle*{3}}
\put(100,50){\line(-1,-1){10}}
\put(100,50){\line(1,-1){10}}
%\put(100,40){\line(0,-1){30}}
\put(100,40){\line(0,1){15}}
\put(100,40){\line(-1,-1){20}}
\put(100,40){\line(1,-1){20}}
\put(100,65){\makebox(0,0){$\vdots$}}

\put(80,45){\makebox{\tiny${\mathbf{0}}$}}
\put(115,45){\makebox{\tiny${\mathbf{c}}$}}
\put(76,11){\makebox{\tiny${\mathbf{m}}$}}
\put(110,11){\makebox{\tiny${\mathbf{c-m}}$}}

\put(80,-5){\makebox{$\mathcal{S}=\langle 4,5\rangle$}}
\put(86,-20){\makebox{planar}}
\put(83,-35){\makebox{MF yes}}

%%%%%%%%%%

\put(200,40){\circle*{3}}
\put(200,50){\circle*{3}}
%\put(100,30){\circle*{3}}
\put(200,40){\circle*{3}}
%\put(100,10){\circle*{3}}
\put(190,40){\circle*{3}}
\put(210,40){\circle*{3}}
%\put(190,30){\circle*{3}}
%\put(180,20){\circle*{3}}
\put(200,20){\circle*{3}}
\put(200,30){\circle*{3}}
\put(200,50){\line(-1,-1){10}}
\put(200,50){\line(1,-1){10}}
%\put(100,40){\line(0,-1){30}}
\put(200,40){\line(0,1){15}}
\put(200,40){\line(0,-1){20}}
%\put(200,40){\line(1,-1){20}}
\put(200,65){\makebox(0,0){$\vdots$}}

\put(180,45){\makebox{\tiny${\mathbf{0}}$}}
\put(215,45){\makebox{\tiny${\mathbf{c}}$}}
\put(195,11){\makebox{\tiny${\mathbf{m}}$}}

%\put(80,0){\makebox{$\mathcal{S}

\put(175,-5){\makebox{$\mathcal{S}=\langle 4,5,6\rangle$}}
\put(175,-20){\makebox{Gorenstein}}
\put(183,-35){\makebox{MF yes}}
%%%%%%%%%%

\put(300,40){\circle*{3}}
\put(300,50){\circle*{3}}
%\put(100,30){\circle*{3}}
\put(300,40){\circle*{3}}
%\put(100,10){\circle*{3}}
\put(290,40){\circle*{3}}
\put(310,30){\circle*{3}}
%\put(190,30){\circle*{3}}
%\put(180,20){\circle*{3}}
\put(300,20){\circle*{3}}
\put(300,30){\circle*{3}}
\put(300,50){\line(-1,-1){10}}
%\put(200,50){\line(1,-1){10}}
%\put(100,40){\line(0,-1){30}}
\put(300,40){\line(0,1){15}}
\put(300,40){\line(0,-1){20}}
\put(300,40){\line(1,-1){10}}
\put(300,65){\makebox(0,0){$\vdots$}}

\put(280,45){\makebox{\tiny${\mathbf{0}}$}}
\put(315,35){\makebox{\tiny${\mathbf{c}}$}}
\put(295,11){\makebox{\tiny${\mathbf{m}}$}}

%\put(80,0){\makebox{$\mathcal{S}

\put(277,-5){\makebox{$\mathcal{S}=\langle 4,5,7\rangle$}}
\put(264,-20){\makebox{ non Gorenstein}}
\put(285,-35){\makebox{MF no}}
%%%%%%%%%%

\put(400,40){\circle*{3}}
\put(400,50){\circle*{3}}
%\put(100,30){\circle*{3}}
\put(400,40){\circle*{3}}
%\put(100,10){\circle*{3}}
\put(390,40){\circle*{3}}
%\put(410,30){\circle*{3}}
%\put(190,30){\circle*{3}}
%\put(180,20){\circle*{3}}
\put(400,20){\circle*{3}}
\put(400,30){\circle*{3}}
\put(400,50){\line(-1,-1){10}}
%\put(200,50){\line(1,-1){10}}
%\put(100,40){\line(0,-1){30}}
\put(400,40){\line(0,1){15}}
\put(400,40){\line(0,-1){20}}
%\put(400,40){\line(1,-1){10}}
\put(400,65){\makebox(0,0){$\vdots$}}

\put(380,45){\makebox{\tiny${\mathbf{0}}$}}
%\put(415,35){\makebox{\tiny${\mathbf{c}}$}}
\put(390,11){\makebox{\tiny${\mathbf{m=c}}$}}

%\put(80,0){\makebox{$\mathcal{S}

\put(370,-5){\makebox{$\mathcal{S}=\langle 4,5,6,7\rangle$}}
\put(368,-20){\makebox{non Gorenstein}}
\put(385,-35){\makebox{MF yes}}

\put(460,40){\makebox(0,0)[0]{$n=0$}}
\put(460,30){\makebox(0,0)[0]{$n=-1$}}
\put(460,20){\makebox(0,0)[0]{$n=-2$}}
%\put(50,60){\makebox(0,0){$\mathfrak{R}:$}}
%\put(90,0){\makebox(0,0){$-3,-1,-2,0,-2$}}
\qbezier[50](50,40)(190,40)(430,40)
\qbezier[50](50,30)(190,30)(430,30)
\qbezier[50](50,20)(190,20)(430,20)
%\qbezier[50](50,10)(190,10)(430,10)

\end{picture}

The examples are propotypes  of four families. The first one is the family of plane curves. In fact, for  {\it  irreducible plane curves} MF was verified in \cite{KNS1}.
 In the present section we will prove that MF is satisfied by non-irreducible plane curves as well.
 
The second root is a representative of non-planar Gorenstein curves. For several such curves  MF still holds, though counterexamples exist as well
(see e.g. Example \ref{ex:122}).

Among  non-Gorenstein germs it is rather easy to find counterexamples to the MF, see the third root above. However, even for certain non-Gorenstein curves MF might hold,
see e.g. the fourth example ---  though this is the extremal case when there exists only one non-zero local minimum point, namely $\mathbf{m}=\mathbf{c}$, and for such germs
MF trivially holds (compare also with Remark \ref{rem:0locmin}).

\end{example}

\begin{ex}\cite{KNS1}\label{ex:122}
Gorenstein property does not necessarily imply MF.
Indeed, in \cite[Example 1.2]{KNS1} we presented a pair of  Gorenstein semigroups of rank one (hence 
a pair of irreducible Gorenstein germs) with different multiplicities but with identical graded roots.

In particular, even for Gorenstein irreducible curves
 the multiplicity cannot be read off from $\bH^*$.

In Proposition \ref{tavolimult} we  show that in the irreducible Gorenstein case  those  local minima which are  responsible for the failure of MF are those which sit in the interval
$(m,2m)$:   if there is only one semigroup element in $(m,2m)$  then MF holds (see Theorem \ref{th:MF}). However if there are more, then MF might fail (as in
the third case of Example \ref{ex:45} or in the just mentioned example from \cite{KNS1}).

\end{ex}

\begin{example}\label{ex:decfail}MF typically does not hold for decomposable singularities (decomposable germs typically are non-Gorenstein).
We exemplify the phenomenon via $(C,o)=C_1\vee C_2$, where both $C_1 $ and $C_2$ are copies of $C_{3,4}$, the plane curve singularity $\{x^3+y^4=0\}$.
The local minima of the weight function of $(C_{3,4},o) $ are  0, $m=3$ and $c=6$ with $w$-values 0, $-1$ and 0. Let us call the local minima for each $C_i$ by $0, m_i, c_i$ ($i=1,2$).
Then, by the   structure of decomposable singularities (see Example \ref{ex:decomp}) we obtain that the non-zero local minima of $(C,o)$ are
$\mathbf{m}=(m_1,m_2)=(3,3)$,  $(m_1,c_2)=(3,6)$, $(c_1,m_2)=(6,3)$ and $\mathbf{c}=(c_1,c_2)=(6,6)$. By (\ref{eq:dec3}) their weight values are $-4,\ -3, \ -3, \ -2$ respectively. In particular,
for any non-zero local minimum point $p\neq \mm$ one has $w(\mathbf{m})< w(p)<0$. This is totally opposite to MF.
The graded root is given  below on the left:

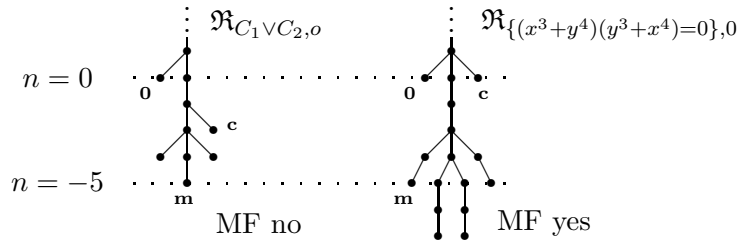
\begin{center}
\begin{picture}(300,110)(15,5)
\linethickness{.5pt}
\put(100,80){\circle*{3}}
\put(90,80){\circle*{3}}
\put(100,90){\circle*{3}}
\put(100,70){\circle*{3}}
\put(100,60){\circle*{3}}
\put(100,50){\circle*{3}}
\put(100,40){\circle*{3}}
\put(110,50){\circle*{3}}
\put(90,50){\circle*{3}}
\put(110,60){\circle*{3}}

\put(100,90){\line(-1,-1){10}}
%\put(100,90){\line(1,-1){10}}
\put(100,95){\line(0,-1){55}}
\put(100,70){\line(1,-1){10}}
\put(100,60){\line(1,-1){10}}
\put(100,60){\line(-1,-1){10}}

\put(100,105){\makebox(0,0){$\vdots$}}

\put(51,80){\makebox(0,0)[0]{$n=0$}}

\put(51,40){\makebox(0,0)[0]{$n=-5$}}
%\put(170,35){\makebox(0,0)[0]{$w_0=3-a$}}
%\put(170,15){\makebox(0,0)[0]{$w_0=1-a$}}

%\put(90,0){\makebox(0,0){$-3,-1,-2,0,-2$}}
\qbezier[20](80,80)(150,80)(220,80)
%\qbezier[70](50,35)(90,35)(130,35)
%\qbezier[10](50,25)(90,25)(130,25)
%\qbezier[10](50,15)(90,15)(130,15)

\put(82,72){\makebox{\tiny${\mathbf{0}}$}}
\put(115,60){\makebox{\tiny${\mathbf{c}}$}}
\put(95,32){\makebox{\tiny${\mathbf{m}}$}}

\put(130,100){\makebox(0,0)[0]{$\mathfrak{R}_{C_1 \vee C_2, o}$}}

\put(127,25){\makebox(0,0)[0]{MF no}}

%%%%%%%%%%

\put(200,80){\circle*{3}}
\put(190,80){\circle*{3}}
\put(210,80){\circle*{3}}
\put(200,90){\circle*{3}}
\put(200,70){\circle*{3}}
\put(200,60){\circle*{3}}
\put(200,50){\circle*{3}}
\put(210,50){\circle*{3}}
\put(190,50){\circle*{3}}
\put(185,40){\circle*{3}}
\put(195,40){\circle*{3}}
\put(205,40){\circle*{3}}
\put(215,40){\circle*{3}}
\put(195,30){\circle*{3}}
\put(205,30){\circle*{3}}
\put(195,20){\circle*{3}}
\put(205,20){\circle*{3}}

\put(200,90){\line(-1,-1){10}}
\put(200,90){\line(1,-1){10}}
\put(200,95){\line(0,-1){45}}
\put(200,60){\line(1,-1){10}}
\put(200,60){\line(-1,-1){10}}
\put(200,50){\line(1,-2){5}}
\put(200,50){\line(-1,-2){5}}
\put(190,50){\line(-1,-2){5}}
\put(210,50){\line(1,-2){5}}
\put(195,40){\line(0,-1){20}}
\put(205,40){\line(0,-1){20}}

\put(200,105){\makebox(0,0){$\vdots$}}

\qbezier[20](80,40)(150,40)(220,40)

\put(182,72){\makebox{\tiny${\mathbf{0}}$}}
\put(210,72){\makebox{\tiny${\mathbf{c}}$}}
\put(178,32){\makebox{\tiny${\mathbf{m}}$}}

%\put(260,80){\makebox(0,0)[0]{$n=0$}}

\put(260,100){\makebox(0,0)[0]{$\mathfrak{R}_{\{(x^3+y^4)(y^3+x^4)=0\},0}$}}

\put(235,25){\makebox(0,0)[0]{MF yes}}

%\put(195,62){\makebox{\tiny{add another root}}}

\end{picture}
\end{center}

This example also shows, that the validity of  MF for $(C_1,o)$ and $(C_2,o)$ does not necessarily imply the validity  for their wedge $(C_1\vee C_2,o)$.

%(The interested reader might draw the $w$-table as well, which shows that the homotopy type of $S_{-1}$ is the circle, hence $\bH^1\not=0$.)

As an opposite to $(C_{3,4}\vee C_{3,4},o)$  we provide above on the right hand side the graded root of the    {\it  plane curve singularity}
 with two branches with  distinct tangent cones, both
isomorphic to $(C_{3,4},o)$.

\end{example}

\subsection{Proof of Multiplicity Formula for plane germs}
Though 
the Gorenstein property does not imply MF, still, for some minima $p$ in a `large' zone 
we have $w(p)\leq w({\bf m})$. 

%Even if we only impose that $(C,o)$ is Gorenstein, some local minima $\elll$ of its weight function already satisfy $w(\elll)\leq w(\mathbf {m})$ (see Proposition \ref{tavolimult}). In Theorem \ref{th:MF} we show that in the case of planar germs this holds for \emph{all} of the local minima and hence the plane curves satisfy MF.

\begin{proposition}\label{tavolimult} Assume that $(C,o)$ is Gorenstein, and
    let $\elll \in \mathcal{S}_{C,o}$ be a local minimum point
    with  $2\mm \leq \elll \leq \mathbf{c}-2\mm$. Then
    $w(\elll)\leq w(\mathbf{m}).$
\end{proposition}

In the proof we need the following Lemma.
%We have the following generalization of Lemma \ref{b-a}:

\begin{lemma} \label{ablak}
    Let $\ell \in \mathbb{Z}_{\geq 0}^r$ and $s \in \mathcal{S}_{C,o}$. (The Gorenstein assumption is not needed here.) Then
    $$w(\ell+s)-w(\ell) \leq w(\ell + s + e_i) - w(\ell + e_i)\ \ \ \mbox{for any $i \in \mathcal{I}$.}$$
\end{lemma}

\begin{proof}
    The formula can be rearranged as
    $$w(\ell+e_i)-w(\ell) \leq w(\ell + s + e_i) - w(\ell + s).$$
    Since both sides are equal to either $-1$ or $+1$ (cf. (\ref{eq:w0valtozasa})), the only nontrivial case is if the left hand side is equal to $+1$. This is equivalent to the existence of a semigroup element $x \in \overline{\Delta}_i(\ell) \not= \emptyset$ but then $x+s \in \overline{\Delta}_i(\ell+s) \not= \emptyset$ too, so the right hand side is equal to $+1$ as well.
\end{proof}

\begin{proof}[Proof of Proposition \ref{tavolimult}]
    First, we show that we may assume that
    \begin{equation}\label{eq:dag1}
    w(\elll-\mathbf{m}) \geq w(\elll).\end{equation}
    This follows from the fact that
     $$(i) \ w(\elll-\mathbf{m}) < w(\elll) \ \ \ \text{ and } \ \ \ (ii) \
      w(\mathbf{c}-\elll-\mathbf{m}) < w(\mathbf{c}-\elll)$$
    cannot hold simultaneously. Indeed, if {\it (i)} holds, then applying Lemma \ref{ablak} several times resulting  a positive shift with vector $\mathbf{m}$, we get that
    $w(\elll) < w(\elll+\mathbf{m})$ too. But this, by Gorenstein symmetry, is
$ w(\mathbf{c}-\elll)<   w(\mathbf{c}-\elll-\mathbf{m}) $, contradicting {\it (ii)}.
    Hence (\ref{eq:dag1}) is true either for $\elll$ or for
    $\elll'=\ccc-\elll$, where by symmetry $\elll'$ is also a local minimum and   $w(\elll)=w(\elll')$.

     Hence, we assume $ w(\elll)\leq w(\elll-\mathbf{m}) $ and thus it is sufficient to show that
     $ w(\elll-\mathbf{m})\leq w(\mathbf{m})$.

     Choose an increasing path $\{x^k\}_{k=0}^t \subset R(\mathbf{m},\elll-\mathbf{m})$ from $\mathbf{m}$ to $\elll-\mathbf{m}$ (here we use the assumption $\elll\geq
2\mathbf {m}$). Define the dual path $\{ \overline{x}^{k}\}_{k=0}^t$ given as $\overline{x}^k := \elll-x^{t-k}$. We now have
    \begin{align*}
        w(\elll-\mathbf{m})-w(\mathbf{m}) &= \sum_{k=0}^t\big( w(x^{k+1})-w(x^k)\big) = \sum_{k=0}^t\big( w(\overline{x}^{k+1})-w(\overline{x}^k)\big).
    \end{align*}
   Then by  Corollary \ref{szimmutak}, an increase in $w$ along one of the two paths implies a decrease in $w$ along the other one, whence the total difference $w(\elll-\mathbf{m})-w(\mathbf{m})$ cannot be positive.
    \end{proof}

%Though  Proposition \ref{tavolimult} holds for any Gorestein germ,
Having in mind the examples listed above
it is hard to find out  the exact limits of MF.   In the next theorem we provide two rather large families, covering
the most important  key applications, including e.g. the plane curves singularities.

In the proof we will use the following analytic statement.
%Before we state the main result of this section, we will prove the following needed statement.
% (Since we will apply it for plane curves, for simplicity we formulate it in this context, but again,  we expect that it is valid more generally.)
%Below $\mathfrak{m}_{C,o}$ denotes the maximal ideal of $\mathcal{O}_{C,o}$.

\begin{proposition}\label{th:Lcond}
Assume that $(C,o)$ is a reduced plane curve singularity  with $mult(C,o)\geq 3$. Then there exists no `linear' function
$L\in\mathfrak{m}_{C,o}\setminus \mathfrak{m}_{C,o}^2$ such that $\mathfrak{v}(L)=\ccc$.
\end{proposition}
\begin{proof} (a) First assume that $r=1$. Fix a Puiseux parametrization
 $ x=t^{m_1}, \ y=t^{n_1} + \cdots$
 with $m_1<n_1$ and $m_1 \nmid n_1$, say $\pi$.
 Then we claim that for any $L\in\mathfrak{m}_{C,o}\setminus \mathfrak{m}_{C,o}^2$
 the intersection multiplicity $(C,L=0)_o$ (what in the sequel will be   abbreviated by     $(C,L)_o$ and  equals  
 $\mathfrak{v}(L)$ )  is an element of the set
  $\{0 ,m_1,2m_1, \dots, km_1, n_1 \}$, where
            $k$ is the maximal integer such that $km_1 < n_1$.     
Indeed, write  $L(x,y)$ as  $ax + by +$ higher order terms.
 If $a \not= 0$ then $\text{ord}_t(L\circ \pi) = m_1$.
        Otherwise we must have $b\not=0$ so $\text{ord}_t(L\circ \pi)\in \{2m_1, \dots, km_1, n_1 \}$.

 The elements of the  set  $\{0 ,m_1,2m_1, \dots, km_1, n_1 \}$  are  bounded by $n_1$, which is an element of the set of minimal generators of
$\mathcal{S}_{C,o}$. On the other hand,  for any $(C,o)$ with $m_1=mult(C,o)\geq 3$,
 it is known that  the largest  element of the minimal set of generators of $\mathcal{S}_{C,o}$ is
 strictly smaller than the Milnor number, which equals the conductor. (For properties of the generators of $\mathcal{S}_{C,o}$ see e.g. \cite{bres,wall,Zar}.) Hence if $r=1$  and  
 $mult(C,o)\geq 3$, then  $\mathfrak{v}(L)<\ccc$ for any 
 $L\in\mathfrak{m}_{C,o}\setminus \mathfrak{m}_{C,o}^2$.

 Notice also that if $r=1$ and $mult(C,o)=2$  then $n_1-1$ is the conductor, i.e. in this case $\mathfrak{v}(L)\leq \ccc+1$ for any $L$.
(If  $mult(C,o)=1$ then $\mathfrak{v}(L) $ can be any positive number.)

 (b) Next assume that $r\geq 2$. Recall that the $i^{th}$ component $c_i$ of $\ccc$ satisfies
 $c_i=\ccc(C_i,o)+(C_i, \cup_{j\not=i}C_j)_o$.
 Hence, if there is a component $(C_i,o)$ with $mult(C_i,o)\geq 2$  then $\ccc(C_i,o) +2 \leq c_i$, since $(C_i,D)_o \geq mult(C_i)\cdot mult(D) \geq 2$ for any $D$. On the other hand,
since both $(C,o)$ and $(C_i,o)$ have embedded dimension 2, the restriction $\mathfrak{m}_{C,o}/
\mathfrak{m}^2_{C,o}
\to \mathfrak{m}_{C_ i,o}/\mathfrak{m}^2_{C_ i,o}$ is an isomorphism
 and thus takes $L$ to $\mathfrak{m}_{C_ i,o}\setminus\mathfrak{m}^2_{C_ i,o}$. In particular,
 $\mathfrak{v}_i(L) \leq \ccc(C_i,o)+1$ for any $L\in\mathfrak{m}_{C,o}\setminus \mathfrak{m}_{C,o}^2$ by part (a), thus $\ccc\not=\mathfrak{v}(L)$.

Hence we can assume that all components are smooth and (since $mult(C,o)\geq 3)$) $r\geq 3$ too.

 Assume indirectly that $\mathfrak{v}(L)=\ccc$. Then $(L,C_1)_o=\mathfrak{v}_1(L)=c_1>(C_1,C_2)_o$.
 Similarly, $(L,C_2)_o>(C_1,C_2)_o$. That is, under this assumption
 \begin{equation*}\label{eq:in}
 (C_1,C_2)_o< \min\{(L,C_1)_o,(L,C_2)_o\}.
 \end{equation*}
 But this fact cannot happen. Indeed, assume that $L$ is given by $x=0$, parametrize
 $(C_1,o)$ by $\pi$ as $x=t^a+\ldots, $ $y=t^b+\ldots$, and let the equation of $(C_2,o)$ be
 $f_2=xg(x,y)+y^dh(y)=0$, $h(0)\not=0$. Then $(L,C_1)_o=a$, $(L,C_2)_o=d$, but  $\text{ord}_t(f_2\circ \pi)\geq \min\{a,d\}$.
 %Therefore, $\mathfrak{v}(L)\not=\ccc$. 
 \end{proof}

 The main result of this section is the following.

\begin{theorem}\label{th:MF} Assume that $(C,o)$ is

(I) either irreducible, Gorenstein and it satisfies the following property:
\begin{equation}\label{eq:dag2}
\mbox{$\calS_{C,o}\cap (m,2m)$ contains at most 1 element;}\end{equation}

 (II) or, $(C,o)$ is an isolated plane curve singularity.

\noindent
Then the Multiplicity Formula holds for $(C,o)$.
\end{theorem}

\begin{proof}

(I) %Assume first that $(C,o)$ is irreducible, Gorenstein, and it satisfies {\it (B)}.
If the multiplicity is $\leq 2$ then MF holds. Hence we can assume that $mult(C,o)\geq 3$.
Then, via Proposition \ref{tavolimult},  we need to check that
$w(\elll)\leq w(m)$ for any local minimum with $m<\elll<2m$. But, by assumption,
  if such a $\elll$ exists, then it is the  unique semigroup element of $(m,2m)$, and
$\hh(\elll)=\hh(m)+1$. Hence $w(\elll)=  w(m)+2+m-\elll$. But $\elll\geq m+2$.
Indeed,  $\elll$ is a local minimum, hence $\elll-1\not\in\calS_{C, o}$,  but $m\in\calS_{C, o}$,  so necessarily  $\elll-1>m$.

(II)
If $r=1$, then $(C,o)$  is Gorenstein, and its semigroup satisfies (\ref{eq:dag2}) (for the structure of the semigroups of
plane curve singularities see e.g. \cite{bres,wall,Zar}). Hence MF is guaranteed by (I).

We proceed by induction on the cardinality $r$ of $\mathcal{I}$.
We assume that MF holds for any plane curve with less than $r$ irreducible components, and we
 consider  a plane curve $(C,o)$ with $r$   components  ($r>1$).
If its multiplicity is $\leq 2$ then there is nothing to prove, hence we assume $mult(C,o)\geq 3$.
Then we need to show that $w(\elll)\leq w({\bf m}$ for any local minimum point 
$\elll=(\elll_1,\ldots, \elll_r)\not\in\{\mathbf {0},  \ccc\}$.

Again, via Proposition \ref{tavolimult} and 
Gorenstein symmetry, 
we can assume that
$\elll \in   R(\mm,\mathbf{c}) \setminus R(2\mm, \mathbf{c})$, $\elll>\mm$.

The point $\elll$, being a semigroup element, is a $\mathfrak{v}$--value. On the other hand,
 for any $g\in \mathfrak{m}_{C,o}^2 $
 % (where $\mathfrak{m}_{C,o}\subset \mathcal{O}_{C,o}$ is the maximal ideal)
 one has  $\mathfrak{v}(g)\geq 2\mathbf{m}$. Thus, our $\elll$ satisfies the following:
\begin{equation}\label{eq:L}\mbox{
there exists $L\in \mathfrak{m}_{C,o}\setminus
\mathfrak{m}_{C,o}^2$ such that $\mathfrak{v}(L)=\elll$. }
\end{equation}
Moreover, since $\elll\not\geq 2\mathbf{m}$,
 there exists some $i\in\mathcal{I}$ such that $m_i\leq \elll_i<2m_i$.

\vspace{2mm}

\underline{(i) Assume that there exists some $i\in\mathcal{I}$ such that  $\elll_i=m_i$.}

\vspace{2mm}

\noindent Let us fix some $i$ with $\elll_i=m_i$.  By re-indexing we can assume that $i=1$.
%this index is 1. %, even more, if there is an index $j$ with $c_j=m_j$ then $c_1=m_1$ as well.
%First we assume that  $c_1=m_1$.
Write $\elll$ as $(\elll_1,\bar{\elll})$, where $\bar{\elll}=(\elll_2, \ldots, \elll_r)$. Below we will use the $\,\bar{\cdot}\,$
notations for invariants associated with the plane curve singularity $(\bar{C},o)=\bigcup_{i>1}(C_i,o)$.

Then $\hh(\elll_1,\bar{\elll})=\hh(0, \bar{\elll})$ by analysing the lack of semigroup elements  in the corresponding region
$ R((0, \bar{\elll}), (m_1-1, \infty))$. %.
But by Proposition \ref{lem:wrestricted}, $\hh(0, \bar{\elll})=\bar{\hh}( \bar{\elll})$. Moreover, analysing the weights of the neighbours of $\elll$ and $\bar{\elll}$,
a similar argument shows that
 if $(m_1,\bar{\elll})$ is a local minimum of $w$ then $\bar{\elll}$ is a local minimum for $\bar{w}$.
Furthermore,
\begin{equation}\label{eq:=}
w((m_1, \bar{\elll}))=
2\hh((m_1, \bar{\elll}))-m_1-|\bar{\elll}|=2\bar{\hh}(\bar{\elll})-m_1-|\bar{\elll}|=\bar{w}(\bar{\elll})-m_1.\end{equation}
Here we wish to  apply induction for $(\bar{C},o)$. Note that $\bar{\elll}$ is a local minimum of $\bar{w}$ with $\bar{\elll}>\overline{\mathbf{m}}$.
In particular, $(\bar{C},o)$ cannot be smooth. If it has multiplicity 2 then $\bar{w}(\bar{\elll})=0=\bar{w}(\overline{\mathbf{m}})$.
Finally, assume that it has multiplicity $\geq 3$.  Then, as in the proof of Proposition \ref{th:Lcond}, the restriction
 $\mathfrak{m}_{C,o}/\mathfrak{m}^2_{C,o}\to \mathfrak{m}_{\bar{C},o}/ \mathfrak{m}^2_{\bar{C},o}$ is an isomorphism,  hence $L\not\in \mathfrak{m}^2_{\bar{C},o}$ and  by this  Proposition, $\bar{\elll}= \bar{\mathfrak{v}}(L)$ is not the conductor of
$(\bar{C},o)$. Thus, by the inductive step, $\bar{w}(\bar{\elll})\leq \bar{w}(\overline{\mathbf{m}})$.

This combined with (\ref{eq:=}) gives
$w(\elll)=w((m_1, \bar{\elll}))=\bar{w}(\bar{\elll})-m_1\leq \bar{w}(\overline{\mathbf{m}})-m_1=w(\mathbf{m})$.

\vspace{2mm}

\underline{(ii) Assume that  $m_j<\elll_j$ for all $j\in \mathcal {I}$.}

\vspace{2mm}

Recall that there exists  an index $i$ with  $m_i< \elll_i<2m_i$.
By property (\ref{eq:dag2}) of irreducible plane curve germs,  if such $\elll_i\in(m_i,2m_i)$ exists, it is the unique semigroup element of $\mathcal{S}_{C_i,o}$
 in $(m_i,2m_i)$. In order to run the induction,
it is convenient to choose the index $i$
such that there exists no semigroup element of  $\mathcal{S}_{C_i,o}$
 in the interval $(m_i,\elll_i)$ (even if $\elll_i=2m_i$). By our discussion such an $i$ exists.

Then,
 after re-indexing, write $i=1$ and use notations
 $(\elll_1,\bar{\elll})$ and $(\bar{C},o)=\bigcup_{i>1}(C_i,o)$ just as in case (i).

The property $\{\elll_j>m_j \ \mbox{for all $j$}\}$ says that the tangent line to the smooth germ $\{L=0\}$ agrees with the tangent line $TC_j$ for all components $(C_j,o)$. Hence all the tangent cones are the same. In particular, if for a certain semigroup element $s=\mathfrak{v}(g)$
one has $s_j=m_j$ for some $j$, then $g$  must be linear and
$g=0$ intersects $TC_j$ transversely. But then  these facts are  true for all the other components as well, hence
$s=\mathbf {m}$. In other words,
\begin{equation}\label{eq:barD}
\overline{\Delta}(\mathbf{m})=\{\mathbf {m}\}.
\end{equation}
The property  (\ref{eq:barD}) will be used in two ways: for the inductive computation of $w(\elll)$ and also
in the verification of the fact that $\bar{\elll}$ is a local minimum for $\bar{w}$ (when it is).

First, from  (\ref{eq:barD}), $\overline{\Delta}_1(m_1,\bar{\elll})=\emptyset$, so there is no 
semigroup element in the  region  
 $ R((0, \bar{\elll}), (p_1-1, \infty))$, hence 
 $\hh(\elll_1,\bar{\elll})=\hh(0, \bar{\elll})$. 
  Similarly as in case (i),  by Proposition \ref{lem:wrestricted}, $\hh(0, \bar{\elll})=
 \bar{\hh}( \bar{\elll})$.
Therefore
\begin{equation}\label{eq:=ii}
w(\elll)=w((\elll_1, \bar{\elll}))=
2\hh((\elll_1, \bar{\elll}))-\elll_1-|\bar{\elll}|=2\bar{\hh}(\bar{\elll})-\elll_1-|\bar{\elll}|
=\bar{w}(\bar{\elll})-\elll_1.\end{equation}
Next, we wish to know whether or not $\bar{\elll}$ is a local minimum for $\bar{w}$.
We split the discussion into two subcases.

{\bf Case (a)} \ Assume that  $(\bar{C},o)$ is not irreducible.
Then using the fact that there is no semigroup element of $\mathcal{S}_{C_1,o}$ in the interval $(m_1,\elll_1)$,
and $\overline{\Delta}_1(m_1,\bar{\elll}-e_j)=\emptyset$ for any $j\not=1$, we get that
$\hh(\elll\pm e_j)=\bar{\hh}(\bar{\elll}\pm e_j)$, hence $\bar{\elll}$ is indeed a local minimum point of $\bar{w}$.
 Then we apply induction as in case (i), hence $\bar{w}(\bar{\elll})\leq \bar{w}(\overline{\mathbf{m}})$.
This combined with (\ref{eq:=ii}) gives
$w(\elll)=\bar{w}(\bar{\elll})-\elll_1\leq \bar{w}(\overline{\mathbf{m}})-\elll_1\leq w(\mathbf{m})$.

{\bf Case (b)} \ Assume that  $(\bar{C},o)$ is irreducible, i.e. $(\bar{C},o)=(C_2,o)$.

Now, if $\elll_2> m_2+1$ then all the arguments of Case (a) can be repeated, hence $w(\elll)\leq w(\mathbf {m})$.

If ($\ddag$) $\elll_2=m_2+1$, then we cannot  conclude that $\elll_2$ is a local minimum of $\bar{w}$.
On the other hand, we notice that
 there exists no semigroup element of  $\mathcal{S}_{C_2,o}$
 in the interval $(m_2,\elll_2)$ (though  $\elll_2=2m_2$ might happen when $m_2=1$). Hence, in the start of case (ii), we can
 choose
 $(C_2,o)$ as the distinguished curve as well. Then, with this choice we run again the arguments of (ii), Case (b), by interchanging the two components,
 and we conclude  $w(\elll)\leq w(\mathbf {m})$  except when the other coordinate also has the pathological equality
 ($\ddag$), namely $p_1=m_1+1$.

 Hence, still we have to analyse independently the following case: $(C,o)$ has two components,
 and for both $\elll_i=m_i+1$. But in this situation we compute directly, using also (\ref{eq:barD}), that
 $\hh(\elll)=2$, hence $w(\elll)=2\hh(\elll)-\elll_1-\elll_2=2-m_1-m_2=w(\mathbf {m})$.
\end{proof}

\begin{remark}(a) Note that the validity of the Multiplicity Formula depends only on the semigroup (in fact it can be formulated for any good semigroup):
$w(\mathbf{m})$ and $\bH^0$ depend only on $\mathcal{S}$. Therefore,  via Theorem \ref{th:MF}(II),  MF holds for any curve singularity $(C,o)$
 whose semigroup is identical with the semigroup of a plane curve singularity.
The same observation is valid for part (I) of the theorem: the listed  assumptions also depend only on the semigroup.

(b) In the proof of Theorem \ref{th:MF}(II)  we had the following situation: $p=\mathfrak{v}(L)$  for a certain $L\in\mathfrak{m}_{C,o}\setminus \mathfrak{m}_{C,o}^2$,
the restriction of $L$ via the {\it isomorphism} $\mathfrak{m}_{C,o}/\mathfrak{m}^2_{C,o}\to \mathfrak{m}_{\bar{C},o}/\mathfrak{m}^2_{\bar{C},o}$  remained `linear', hence the identity
$\bar{p}=\bar{\mathfrak{v}}(L)$ (and Proposition \ref{th:Lcond} applied for $(\bar{C},o)$)  implied that $\bar{p}$ cannot be the conductor of  $(\bar{C},o)$.

Here the fact that $\mathfrak{m}_{C,o}/\mathfrak{m}^2_{C,o}\to \mathfrak{m}_{\bar{C},o}/\mathfrak{m}^2_{\bar{C},o}$  is an isomorphism (i.e. both germs $(C,o)$ and $(\bar{C},o)$ are planar)
is crucial.
Indeed, in  Example \ref{ex:decfail}  $p=(m_1,c_2)$ is realized as $\mathfrak{v}(L)$ for a  linear $L$, however the projection of $L$ is in 
 $\mathfrak{m}^2_{C_2,o}$, and even more, $c_2$ is the conductor of $(C_2,o)$.

This phenomenon (together with the lack of property (\ref{eq:barD})) obstructs the inductive step in the (present)  proof of the theorem for more general non-planar germs.
\end{remark}

\section{Third  application. The nonpositivity of the $\bH^*_{\text{red}}$ support}

\subsection{Statement of the theorem} Recall that the reduced lattice cohomology of a reduced  curve singularity $(C,o)$ has the form
$$\mathbb{H}_{\text{\normalfont red}}^q(C,o) = \bigoplus_{n\in\mathbb{Z}}\mathbb{H}_{\text{\normalfont red},\, 2n}^q(C,o).$$
In this section we will prove the following structure theorem:

\begin{theorem}[Nonpositivity theorem]\label{th:Main}
    The weight-grading of the \textit{reduced} lattice cohomology  is \textit{nonpositive}, i.e.
    $\mathbb{H}_{\text{\normalfont red},\, 2n}^q(C,o) = 0$
    for all $q\geq 0$ and $n > 0$. In other words,
\begin{equation*}
    \mathbb{H}_{\text{\normalfont red}}^q(C,o) = \bigoplus_{n\leq 0}\mathbb{H}^q_{\text{\normalfont red},\, 2n}(C,o).
\end{equation*}
\end{theorem}

In fact, we will prove a stronger version of Theorem \ref{th:Main}, namely:

\begin{theorem}\label{th:Main2}
    The spaces $\{S_n\}_{n \in \bZ}$ used to define $\bH^* (C,o)$ are contractible for all $n>0$.
\end{theorem}

The proof of Theorem \ref{th:Main2} follows the next strategy:
we construct a strong deformation  retraction from
the (contractible) space $R(\mathbf{0},\ccc)$ to its subspace $S_{n}\cap R(\mathbf{0},\ccc)$
 for any  $n>0$. This is done inductively
 via a strong deformation retraction  from  $S_n\cap R(\mathbf{0},\ccc)$
 to  $S_{n-1}\cap R(\mathbf{0},\ccc)$, which will be the composition of repeatedly
 choosing `independent'  lattice points
 $\ell \in S_n\cap R(\mathbf{0},\ccc)$ of maximal weight $n\geq 2$ and then deforming all cubes containing $\ell$ into $S_{n-1}\cap  R(\mathbf{0},\ccc)$, in this way  `eliminating' $\ell$.

\subsection{The M-vertices of cubes}\label{ss:6.2}

For cubes we used several notations, namely $(\ell,J)$, or $R(\ell, \ell+e_J)$, where $e_J:=\sum_{i\in J}e_i$. In this way we basically distinguish the vertices $\ell$ and $\ell+e_J$. However, in certain discussions it is convenient to distinguish some other vertices, e.g. a  vertex $\ell+e_{K}$ ($K\subset J$)  identified by
$w(\ell+e_{K})=\max\{w(\ell')\,|\, \ell' \ \text{is a vertex of } \ (\ell,J)\}$.
Then we can redefine the cube $(\ell, J)$, viewed from $\ell+e_K$ in `positive and negative' coordinate directions, see below.

\begin{nota}
For a lattice point $\ell \in \bZ^r_{\geq 0}$ and index sets $J^+, J^- \subset \mathcal{I}$ with $ J^+ \cap J^- = \emptyset$ and $\ell\geq e_{J^-}$, let $(\ell, J^+, J^-)$ denote the cube with vertices $ \ell + e_{K^+} - e_{ K^-}$, where $K^+$ and  $K^-$ run over all subsets of $ J^+$ and $J^-$ respectively. In this case the dimension of the cube is $q=\#\{J^+ \cup J^-\}$.
\end{nota}

%\begin{remark} \label{rem:facesasintersections}
 %   Notice that the faces of a cube $(\ell, J^+, J^-)$ may be obtained by subsequently
% intersecting it with affine subspaces of the form $\{ x_i= \ell_i\}$ or $\{ x_i =
%\ell_i \pm 1\}$ where $i \in J^{\pm}$.
%\end{remark}

\begin{define}\label{def:goodstarcube}
     The lattice point  $\ell$  is called  \textit{an M-vertex of the cube}  $(\ell, J^+, J^-)$ if
    \begin{itemize}
        \item[(i)] $w_0(\ell+e_i) < w_0(\ell)$ for all $i \in J^+$,  and
        \item[(ii)] $w_0(\ell-e_i) < w_0(\ell)$ for all $i \in J^-$.
    \end{itemize}
\end{define}
\noindent

A lattice point having maximal weight in a cube is an $M$-vertex of that cube. The lemma below states the converse.

%The next lemma connects this definition with the comment  of the first paragraph of this subsection.

\begin{lemma} \label{lem:goodstarcube}
     If $\ell$ is an M-vertex of $(\ell, J^+, J^-)$ then $w(\ell') \leq w(\ell)$
     for any vertex $\ell'$ of $(\ell, J^+, J^-)$.
\end{lemma}

\begin{proof}
%    Notice that
 %   \begin{itemize}
        (I) \ Part (i) of Definition \ref{def:goodstarcube}
        is equivalent with $\hh(\ell)=\hh(\ell+e_i)$ for any  $i \in J^+$. Then induction on the cardinality $\#K^+$, for $K^+\subset J^+$, and
            (\ref{eq:matroid}) imply $\hh(\ell)=\hh(\ell+e_{J^+})$. Hence
             $w(\ell)-w(\ell+e_{J^+})=\#J^+$.

        (II) \   Similarly, (ii) reads as
         $\hh(\ell-e_i)=\hh(\ell)-1$ for any  $i \in J^-$. Then, induction on the cardinality $\#K^-$, for
         $K^-\subset J^-$, and
            (\ref{eq:matroid}) show that $\hh(\ell-e_{J^-})=\hh(\ell)-\#(J^-)$. Hence
             $w(\ell)-w(\ell-e_{J^-})=\#J^-$.

          (III) \
        Assume that there exists  a vertex $\ell'$ of  $(\ell, J^+, J^-)$ such that
        $w(\ell')>w(\ell)$. Then take an increasing path $\{ x^k\}_{k=0}^t$ of length $t=\#J^-+\#J^+$, so that $x_0=\ell-e_{J^-}$, $x_t=\ell+e_{J^+}$, and for a certain
        $k'$ we have $x^{k'}=\ell'$. Then
        \begin{equation}\tag{a}
            w( x^{k'})-w(\ell-e_{J^-}) > w( \ell)-w(\ell-e_{J^-})=\#J^-,
        \end{equation}
        \begin{equation}\tag{b}
            w( x^{k'})-w(\ell+e_{J^+}) > w( \ell)-w(\ell+e_{J^+})=\#J^+.
        \end{equation}
        Then    (\ref{eq:w0valtozasa}) and (a) (resp. (b)) show that
        from $\ell-e_{J^-}$ to $ x^{k'}$ we need strictly more than  $\#J^-$ steps in the path, respectively
         from $ x^{k'}$ to $\ell+e_{J^+}$  we need strictly  more than  $\#J^+$ steps. Hence the
         total number of steps in the  path is $>\#J^-+\#J^+=t$, which is a contradiction.
\end{proof}

\subsection{Good directions associated with lattice points}
In this subsection we  prove that for any lattice point $\ell$ with $w_0(\ell) \geq 2$ we can find a coordinate direction along which a (weight decreasing) strong  deformation retraction can
be constructed.
%We will need the following generalization of the second statement  of  Proposition \ref{cl:glm}.

\begin{lemma}\label{lem:wl+cube}
    Let $\ell \geq \mathbf{0}$ be a lattice point and consider a positively directed cube $(\ell, J^+, \emptyset) \subset \bR_{\geq 0}^r$ (for some index set $J^+ \subset \mathcal{I}$). Suppose that
    \begin{itemize}
        \item[(a)] $w(\ell) > w(\ell+e_i)$ for all $i \in J^+$ (i.e. $\ell$ is an M-vertex of
        the cube), and %$(\ell, J^+, \emptyset)$ is a good star cube) and
        \item[(b)] $\ell^+:=\ell + e_{J^+}$ is a generalized local minimum (cf. Lemma \ref{lem:lm}).
    \end{itemize}
    Then $w(\ell) \leq 1$.
\end{lemma}

\begin{proof}
   Let us consider an increasing path $\gamma=\{ x^k\}_{k=0}^t$ from $x^0=\mathbf{0}$ to $x^t=\ell^+$ through $x^{t-\#J^+}=\ell$ (clearly $t \geq \#J^+$).
    Using (a) and  induction (as in part (I) of the proof
    of Lemma \ref{lem:goodstarcube}) we have ($\dag$) $w(\ell+e_{K^+})=w(\ell)-\#K^+$ for any $K^+\subset J^+$. This implies that
    \begin{equation*} \label{eq:wdecreasesingsc}
     w(x^{k+1})<w(x^k)  \ \ \hbox{for any index } \ \ t-\#J^+ \leq k \leq t-1.
    \end{equation*}
Then the first $\#J^+$ steps  of the dual path $\{\ell^+-x^{t-k}\}_k$ connect
$\mathbf {0}$ with $e_{J^+}$.
Along this path the first step is $w$-increasing, but all the others are $w$-decreasing,  cf. Remark \ref{rem:h|R(0,E)}.
Hence, by  (\ref{eq:w(glm)}) we obtain  $w(\ell^+)\leq 1-\#J^+$. Finally, by ($\dag$) we have
$w(\ell) = w(\ell^+)+\#J^+$, hence $w(\ell)\leq 1$.
\end{proof}

\begin{theorem} \label{th:gooddirection}
    Let $\ell$ be a lattice point with $w(\ell) \geq 2$. Then there exists a coordinate direction $i \in \mathcal{I}$, such that for any cube $\square_q=(\ell, J^+, J^-)$ admitting
    $\ell$ as an M-vertex and satisfying $i \notin J^+ \cup J^-$ (i.e. $\square_q$ is contained in the hyperplane $\{ x_i = \ell_i\}$), we have  \begin{equation}\label{eq:kell}
        w(\ell'-e_i )<w(\ell')
    \end{equation}
    for any vertex $\ell'$ of $\square_q$. In addition, this coordinate direction $i$ also satisfies the following:
    \begin{equation}\label{eq:321irany}
        w(\ell-e_i) < w(\ell) < w(\ell+e_i).
    \end{equation}
    Such a direction will be called a good direction of the vertex $\ell$.
\end{theorem}

\begin{proof}
    Set $I^+:=\{ i \in \mathcal{I}\,\vert \, w(\ell + e_i) < w(\ell)\}$. We will argue that in the absence of a good coordinate direction predicted  by the theorem, we would have at  $\ell^+:= \ell + e_{I^+}$ a generalized local minimum and could use Lemma \ref{lem:wl+cube} for the positively oriented cube $(\ell, I^+, \emptyset)$ to get $w(\ell) \leq 1$, which contradicts the assumption.

 First notice that
 by the stability property (\ref{eq:genstability-}) from $w(\ell + e_i) < w(\ell)$ we get
 \begin{equation}\label{eq:I}
 w(\ell^+)<w(\ell^+-e_i) \ \ \mbox{ for every index $i \in I^+$.}
 \end{equation}
 On the other hand, if we set $I^-:=\{ i \in \mathcal{I}\,\vert \, w(\ell - e_i) > w(\ell)\}$, then using the stability property again,
 we get
 \begin{equation}\label{eq:II}
 w(\ell^+)<w(\ell^+-e_i) \ \ \mbox{ for every index $i \in I^-$. }
 \end{equation}
 Now $I^+\cup I^-=\mathcal{I}$ cannot happen since in that case $\ell^+$ would be a generalized local minimum with $w(\ell)> 1$, contradicting Lemma \ref{lem:wl+cube} used for the positively oriented cube $(\ell, I^+, \emptyset)$. On the other hand, for any
  $i \in \mathcal{I} \setminus (I^+ \cup I^-)$  we automatically have
    \begin{equation*}
        w(\ell-e_i) < w(\ell) < w(\ell+e_i).
    \end{equation*}
    We claim that  at least one of the `remaining' coordinates  $i \in \mathcal{I} \setminus (I^+ \cup I^-)$
    satisfy the needed requirements.

    Assume on the contrary that for every coordinate index  $i \in \mathcal{I} \setminus (I^+ \cup I^-)$
    there exists a  cube $\square_{q,i}=(\ell, J^+, J^-)$ with $i \notin J^+ \cup J^-$, admitting $\ell$ as an M-vertex, and having a  vertex $\ell'$ with
     \begin{equation}\label{eq:6.3}
     w(\ell'-e_i)>w(\ell').\end{equation} (Note that they cannot be equal, see (\ref{eq:w0valtozasa})).
    Then clearly, $J^+ \subset I^+$ (because $\ell$ is an M-vertex),
    %and $J^-\cap I^+=\emptyset$,
    hence
    $\ell'\leq \ell+e_{J^+}\leq \ell+e_{I^+}= \ell^+$.
    %, and also
    %$\ell^+-\ell'=E_K$ for some $K\subset \mathcal{I}$ with $i\not\in K$.
    Then, applying the stability property to (\ref{eq:6.3}),
    %$w_0(\ell + \sum_{j \in J^+}e_j-e_i )>w_0(\ell+\sum_{j \in J^+}e_j)$, even more,
\begin{equation}\label{eq:III}
 w_0(\ell')=w_0(\ell'-e_i)-1\ \ \Rightarrow\  \ w_0(\ell^+)=w_0(\ell^+-e_i)-1.
    \end{equation}
    But if this is true for every remaining coordinate $i \in \mathcal{I} \setminus (I^+ \cup I^-)$, then via (\ref{eq:I}), (\ref{eq:II}) and (\ref{eq:III}) the lattice point $\ell^+$ is a generalized local minimum, a contradiction. Therefore, there must exist a good coordinate direction $i \in \mathcal{I} \setminus (I^+ \cup I^-)$, which  satisfies
    (\ref{eq:321irany}) and  the needed properties for any $\square_q$
    contained in the hyperplane $\{ x_i = \ell_i\}$.
\end{proof}

\subsection{The deformation theorem} In this subsection we sketch the proof of the Deformation Theorem and present how it implies the Nonpositivity Theorem. The rigorous technical discussion of the deformation will be given in the following subsection.

\begin{theorem}[Deformation Theorem] \label{th:DefThm}
    Let $(C, o)$ be a reduced  curve singularity. Then for any $n \geq 1$ there exists a strong deformation retraction
    \begin{equation*}
        R(\mathbf{0}, \mathbf{c}) \searrow S_n \cap R(\mathbf{0}, \mathbf{c}).
    \end{equation*}
In particular,   $ R(\mathbf{0}, \mathbf{c}) $ and
$ S_n \cap R(\mathbf{0}, \mathbf{c})$ are homotopy equivalent and  both are contractible.
\end{theorem}

\begin{cor}\label{cor:Sncontr}
    Theorems \ref{th:EUcurves} and \ref{th:DefThm} imply that for any $n \geq 1$, the cubical complex $S_n$ is contractible. Consequently, this proves  the Nonpositivity Theorem (Theorem \ref{th:Main}).
\end{cor}

\begin{proof}[Proof of the Deformation Theorem]

We will give a sequence of strong deformation retractions
\begin{equation*}
    R(\mathbf{0}, \mathbf{c}) \overset{r_{k}}{\searrow} S_{k-1} \cap  R(\mathbf{0}, \mathbf{c}) \overset{r_{k-1}}{\searrow} \ldots S_{n} \cap R(\mathbf{0}, \mathbf{c}) \overset{r_n}{\searrow } S_{n-1} \cap R(\mathbf{0}, \mathbf{c}) \ldots  \overset{r_3}{\searrow} S_2 \cap R(\mathbf{0}, \mathbf{c}) \overset{r_2}{\searrow } S_1 \cap R(\mathbf{0}, \mathbf{c}),
\end{equation*}
where $k=\max\{ w(\ell) \,\vert \, \ell \in R(\mathbf{0}, \mathbf{c}) \cap \bZ^r\}$ (thus
$R(\mathbf{0}, \mathbf{c}) \subset S_k$).

Even more, each $r_n$ above will itself be the composition of strong deformation retractions, one for each lattice point in $(S_{n} \setminus S_{n-1}) \cap R(\mathbf{0}, \mathbf{c})$. Let us denote this set by
\begin{equation*}
    W_{n}:=\{ \ell \in \bZ^r\,\vert \,\ell \in  (S_{n} \setminus S_{n-1}) \cap R(\mathbf{0}, \mathbf{c})\} = \{ \ell \in \bZ^r \cap R(\mathbf{0}, \mathbf{c})\,\vert \, w(\ell) = n\}.
\end{equation*}
Also, let us order this set in any convenient way, such that $W_n=\{ \ell^1, \ldots, \ell^{\#W_n}\}$ and denote the closed cubical complex spanned by $(S_{n-1} \cap R(\mathbf{0}, \mathbf{c})) \cup \{\ell^1, \ell^2, \ldots, \ell^m\}$ by $R_n^m$ (i.e. this contains every closed cube, whose vertex set is contained in $(S_{n-1} \cap R(\mathbf{0}, \mathbf{c}) \cap \bZ^r) \cup \{\ell^1, \ell^2, \ldots, \ell^m\}$). Then, in the next subsection we will construct a sequence of strong deformation retractions
\begin{equation*}
S_{n} \cap  R(\mathbf{0}, \mathbf{c}) = R_n^{\#W_n} \overset{r^{\#W_n}_{n}}{\searrow} R_n^{\#W_n-1} \searrow \ldots R_n^m \overset{r^m_n}{\searrow } R_n^{m-1} \ldots  \overset{r^2_n}{\searrow} R_n^1 \overset{r_n^1}{\searrow } R_n^0= S_{n-1} \cap R(\mathbf{0}, \mathbf{c}).
\end{equation*}

The strong deformation retraction $r_n^m$ shall eliminate the lattice point $\ell^m$ and every closed cube containing it from the cubical complex $R_n^m$. Let us denote the closed cubical complex consisting of those cubes which have $\ell^m$ as their vertex
(the \textit{`closed star' of} $\ell^m$) by
\begin{equation}\label{eq:St_ldef}
    St_{\ell^m}:= \bigcup_{\substack{\square_q \subset R_n^m \\ \ell^m \text{ vertex of }\square_q}}\square_q.\end{equation}
    Then $ St_{\ell^m} \cup R_n^{m-1} = R_n^m$.
Note that the `open star' of $\ell^m $ is $St_{\ell^m}^{\circ}= \bigcup \{\, \square_q^{\circ} \ | \ \square_q \subset R_n^m \text{ and } \ell^m \text{ is a} $ $\text{vertex of }\square_q \, \}$, with $ St_{\ell^m}^{\circ} =  R_n^m \setminus R_n^{m-1} $. 

Let us denote by $T_{\ell^m}$ the cubical subcomplex of $St_{\ell^m}$ consisting of those closed faces in $St_{\ell^m}$, which do not contain $\ell^m$ as a vertex (i.e. $T_{\ell^m}:=\bigcup_{ \square_q \subset St_{\ell^m}, \ \ell^m \notin \square_q} \square_q$). Then clearly $St_{\ell^m} \cap R_n^{m-1} = T_{\ell^m}$. Therefore, we would need a strong deformation retraction $St_{\ell^m} \searrow T_{\ell^m}$, which then could be extended to $R_n^m$  with the identity map on $R_n^{m-1}$, to get a continuous strong deformation retraction $r_n^m: R_n^m \searrow R_n^{m-1}$.
This desired strong deformation retraction $St_{\ell^m} \searrow T_{\ell^m}$ is constructed in the following subsection (cf. Corollary \ref{rem:phiextended}).
\end{proof}

In order for the construction to work, we need the following observation:

\begin{lemma}
    For a good coordinate direction $i$ corresponding to the lattice point $\ell^m$, with $w_0(\ell^m)=n\geq2$ (cf. Theorem \ref{th:gooddirection})\begin{equation}\label{eq:Stesell-Ei}
    St_{\ell^m}=\bigcup_{\substack{(\ell^m, J^+, J^-) \subset R_n^m \\ i \in J^-}}(\ell^m, J^+, J^-)
\end{equation}
 as a topological subspace.
\end{lemma}

\begin{proof}
As $R_n^m \subset S_n$ and $w(\ell^m) =n$, (\ref{eq:w0valtozasa}) implies that
$\ell^m$ is an M-vertex of every 
 $q$--cube $\square_q \subset St_{\ell^m}$ containing $\ell^m$ (compare with Lemma \ref{lem:goodstarcube}).

As $n \geq 2$, we can use Theorem \ref{th:gooddirection}, which implies the existence of a good coordinate direction $i \in \mathcal{I}$.
Using the fact that a cube is contained in $S_n$ if and only if all its vertices $\ell'$ 
satisfy $w(\ell')\leq n$, we conclude that
there are no cubes $(\ell^m, J^+, J^-)$ in $St_{\ell^m}$ with $i \in J^+$, because $w(\ell^m+e_i) > w(\ell^m) = n$.

On the other hand, part (\ref{eq:kell}) of Theorem \ref{th:gooddirection} also implies that if $(\ell^m, J^+, J^-) \subset St_{\ell^m}$ with $i \notin J^+ \cup J^-$, then
$(\ell^m, J^+ , J^- \cup \{ i \}) \subset St_{\ell^m}$ as well. Indeed,  for every vertex $\ell'$ of $(\ell^m, J^+, J^-)$
\begin{equation*}
    w(\ell'-e_i) < w(\ell') \leq n, \ \ \mbox{hence } \ \  \ell'-e_i \in S_{n-1} \subset R_n^m.
\end{equation*}
Therefore, as a topological space $St_{\ell^m}$ is a union of closed cubes containing both $\ell$ and $\ell-e_i$.
\end{proof}

\subsection{Deformation of a closed cube} In this subsection we fill out the technical details of the previous subsection.
For $\ell$ with $w(\ell)\geq 2$ we fix a good direction $i$ and  
we will construct a strong deformation retraction of any union of cubes, containing both $\ell$ and $\ell -e_i$, to their faces furthest from $\ell$. (This vertex $\ell$ plays the role of some $\ell^m$ from the previous subsection.) At the first step we construct the deformation of one such cube and by the properties of the good direction and the virtue of the construction
all of these deformations glue together. 
%Then finally,  iterating this process associated with several points $\ell^m\in W_n$ provides the wished deformation. 

\bekezdes Let $\ell$ be a vertex of a cube $\square_q$, such that there exists a coordinate $i$ with $\ell-e_i$ also a vertex of $\square_q$ (i.e. $\square_q= (\ell, J^+, J^-)$ with $i \in J^-$). We want to construct a strong deformation retraction from the cube $\square_q$ to the cubical subcomplex $T(\square_q)$ consisting of all those faces (subcubes), for which $\ell$ is not a vertex of:
\begin{equation}\label{eq:Tdef}
    T(\square_q):=\bigcup_{\substack{\square_{q'} \text{ face of } \square_q \\ \ell \notin \square_{q'}}} \square_{q'}.
\end{equation}
We will use a central homothety from the lattice point $\ell + e_i$.

\begin{remark}\label{rem:T-belipontok}
We can characterize the generic real points $x$ of the $(q-1)$-dimensional  faces of
$\square_q$ sitting in $T$  by having a coordinate $j$ for which $x_j \in \bZ$, but $x_j \neq \ell_j$ (more precisely,  $x_j$ can be $\ell_j \pm 1$,  the sign depending on whether $j \in J^+$ or $j \in J^-$).
\end{remark}

%\end{document}

\begin{define}
    Let us define the following  family of `homotheties'  with center $\ell+e_i$:
    \begin{equation*} \label{eq:varphidef}
        \varphi: [0, 1] \times \square_q \rightarrow \bR^r, \ \varphi(t, x)=\varphi_t(x)=\varphi_t\Bigg(\sum_j x_j e_j\Bigg) = x + t\Bigg(\sum_{j \neq i}\dfrac{x_j-\ell_j}{\ell_i + 1 - x_i}e_j - e_i\Bigg).
    \end{equation*}

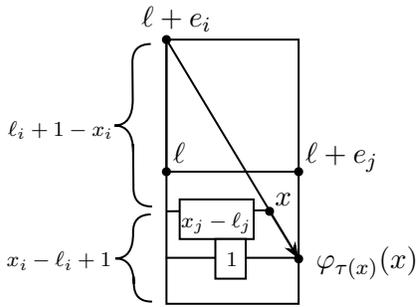
\begin{multicols}{2}

\tikzset{every picture/.style={line width=0.75pt}} %set default line width to 0.75pt
\begin{tikzpicture}[x=1pt,y=1pt,yscale=-1,xscale=1]
%uncomment if require: \path (0,300); %set diagram left start at 0, and has height of 300

%Shape: Square [id:dp2649142155458768]
\draw   (100,105) -- (150,105) -- (150,155) -- (100,155) -- cycle ;
%Shape: Square [id:dp22683124564047064]
\draw   (100,55) -- (150,55) -- (150,105) -- (100,105) -- cycle ;
%Straight Lines [id:da8391560474139537]
\draw    (100,55) -- (149.73,137.6) ;
%Straight Lines [id:da9458737575985436]
\draw    (100,120) -- (138.73,119.6) ;
%Straight Lines [id:da2715954767515625]
\draw    (99.73,137.6) -- (149.73,137.6) ;
%Shape: Brace [id:dp9304851780163446]
\draw   (94.07,56.67) .. controls (89.4,56.7) and (87.09,59.05) .. (87.12,63.72) -- (87.24,77.47) .. controls (87.29,84.14) and (84.99,87.49) .. (80.32,87.53) .. controls (84.99,87.49) and (87.35,90.8) .. (87.4,97.47)(87.38,94.47) -- (87.51,111.22) .. controls (87.55,115.89) and (89.9,118.2) .. (94.57,118.17) ;
%Shape: Brace [id:dp5410353760068816]
\draw   (94.07,121.17) .. controls (89.54,121.24) and (87.31,123.53) .. (87.38,128.06) -- (87.38,128.06) .. controls (87.47,134.53) and (85.26,137.8) .. (80.73,137.87) .. controls (85.26,137.8) and (87.57,141.01) .. (87.67,147.48)(87.63,144.56) -- (87.67,147.48) .. controls (87.74,152.01) and (90.04,154.24) .. (94.57,154.17) ;
%Straight Lines [id:da4854091005884975]
\draw[-{Stealth}]    (139,120) -- (150,138) ;
\draw [shift={(139,120)}, rotate = 58.57] [color={rgb, 255:red, 0; green, 0; blue, 0 }  ][fill={rgb, 255:red, 0; green, 0; blue, 0 }  ][line width=0.75]      (0, 0) circle [x radius= 1.34, y radius= 1.34]   ;
\draw [shift={(150,138)}, rotate = 58.57] [color={rgb, 255:red, 0; green, 0; blue, 0 }  ][fill={rgb, 255:red, 0; green, 0; blue, 0 }  ][line width=0.75]      (0, 0) circle [x radius= 1.34, y radius= 1.34]   ;
%Straight Lines [id:da5735199090279428]
\draw    (100,55) -- (100,105) ;
\draw [shift={(100,105)}, rotate = 90] [color={rgb, 255:red, 0; green, 0; blue, 0 }  ][fill={rgb, 255:red, 0; green, 0; blue, 0 }  ][line width=0.75]      (0, 0) circle [x radius= 1.34, y radius= 1.34]   ;
\draw [shift={(100,55)}, rotate = 90] [color={rgb, 255:red, 0; green, 0; blue, 0 }  ][fill={rgb, 255:red, 0; green, 0; blue, 0 }  ][line width=0.75]      (0, 0) circle [x radius= 1.34, y radius= 1.34]   ;
\draw [shift={(150,105)}, rotate = 90] [color={rgb, 255:red, 0; green, 0; blue, 0 }  ][fill={rgb, 255:red, 0; green, 0; blue, 0 }  ][line width=0.75]      (0, 0) circle [x radius= 1.34, y radius= 1.34]   ;

% Text Node
\draw (39,84.73) node [anchor=north west][inner sep=0.75pt]  [font=\scriptsize] [align=left] {$\displaystyle \ell _{i} +1-x_{i}$};
% Text Node
\draw (38.5,134.23) node [anchor=north west][inner sep=0.75pt]  [font=\scriptsize] [align=left] {$\displaystyle x_{i} -\ell _{i} +1$};
% Text Node
\draw (140,112) node [anchor=north west][inner sep=0.75pt]  [font=\normalsize] [align=left] {$\displaystyle x$};
% Text Node
\draw (101.5,93) node [anchor=north west][inner sep=0.75pt]  [font=\normalsize] [align=left] {$\displaystyle \ell $};
\draw (151.5,93) node [anchor=north west][inner sep=0.75pt]  [font=\normalsize] [align=left] {$\displaystyle \ell +e_j$};
% Text Node
\draw (89.5,41.5) node [anchor=north west][inner sep=0.75pt]  [font=\normalsize] [align=left] {$\displaystyle \ell +e_{i}$};
% Text Node
\draw  [draw opacity=0][fill={rgb, 255:red, 255; green, 255; blue, 255 }  ,fill opacity=1 ]  (105,115.5) -- (133,115.5) -- (133,130.5) -- (105,130.5) -- cycle  ;
\draw (104.5,119.5) node [anchor=north west][inner sep=0.75pt]  [font=\scriptsize] [align=left] {$\displaystyle x_{j} -\ell _{j}$};
% Text Node
\draw  [draw opacity=0][fill={rgb, 255:red, 255; green, 255; blue, 255 }  ,fill opacity=1 ]  (118.5,130.5) -- (130,130.5) -- (130,145.5) -- (118.5,145.5) -- cycle  ;
\draw (121.5,134.5) node [anchor=north west][inner sep=0.75pt]  [font=\scriptsize] [align=left] {$\displaystyle 1$};
% Text Node
\draw (155.5,132) node [anchor=north west][inner sep=0.75pt]  [font=\normalsize] [align=left] {$\displaystyle \varphi _{\tau ( x)}( x)$};

\end{tikzpicture}

\hspace{-25mm}One can easily check that if $\varphi_{t_1}(x) \in \square_q$, then    \begin{equation*}  \hspace{-20.5mm}\varphi_{t_1+t_2}(x)= \varphi_{t_2}(\varphi_{t_1}(x)),
\end{equation*}
\hspace{-20.5mm}i.e. $\varphi$ respects the \emph{group law}. For a given point $x \in \square_q$ set
    \begin{equation*}
        \hspace{-20.5mm}\tau_j(x):= \begin{cases} x_i - (\ell_i-1), \hspace{21mm} \text{ if }j=i;\\
        \infty, \hspace{37.5mm} \text{ if }x_j=\ell_j, i \neq j; \\
        (\ell_i+1-x_i) \dfrac{1-|x_j-\ell_j|}{|x_j-\ell_j|}, \ \text{otherwise}.
        \end{cases}
    \end{equation*}
\end{multicols}

    By this definition, if $x_j \neq \ell_j$ ($j\neq i$), then $\varphi_{\tau_j(x)}(x)$ has $j$-coordinate $\ell_j + {\rm sgn}(x_j-\ell_j) \in \bZ$, whereas $\varphi_{\tau_i(x)}(x)$ has $i$-coordinate $\ell_i-1$. Then set
    \begin{equation}\label{eq:taudef}
        \tau(x):=\min \{\tau_j(x)\,\vert \, j \in \{1, \ldots, r\}\},
    \end{equation}
    i.e. the smallest $t \in [0, 1]$ such that $\varphi_{t}(x)\in T(\square_q)$. Notice, that this characterization implies that if $\varphi_t(x) \in \square_q$, then by the group law
    \begin{equation}\label{eq:tauesphi}
        \tau (\varphi_t(x))=\tau(x)-t;
    \end{equation}
    thus for any point $x\in \square_q$ the property $x \in T(\square_q)$ is equivalent to $\tau(x)=0$ (use Remark \ref{rem:T-belipontok}).

    We can now define the desired strong deformation retraction $\phi^{\square_q, i}: \square_q \searrow T(\square_q)$:
    \begin{equation} \label{eq:phidef}
        \phi^{\square_q, i}: [0, 1] \times \square_q \rightarrow \square_q, \ \phi^{\square_q, i}_t(x) = \begin{cases}
            \varphi_t(x),  \hspace{8mm} \text{ if }t \leq \tau(x);\\
            \varphi_{\tau(x)}(x), \hspace{3.5mm} \text{ if }t \geq \tau(x).
        \end{cases}
    \end{equation}
    Clearly,  points $x \in T(\square_q)$ have $\phi^{\square_q, i}_t(x) \equiv x$ for any $t$. On the other hand for any $x\in \square_q$ we have $\tau(x) \leq 1$ ($\tau_i \leq 1$ on $\square_q$), so $\phi^{\square_q, i}_1(x)=\varphi_{\tau(x)}(x)$ which is indeed contained in $T(\square_q)$ as $\tau(\varphi_{\tau(x)}(x))=0$, see (\ref{eq:tauesphi}). Thus we indeed defined a strong deformation retraction $\square_q \searrow T$.
\end{define}

\begin{claim}\label{cl:defretr}
    Let $\ell\in \mathbb{Z}^r$ be a lattice point, $i \in \mathcal{I}$ a coordinate and $\overline{St} \subset \mathbb{R}^r$ a union of closed cubes containing both $\ell$ and $\ell-e_i$ as vertices:
    \begin{equation*}
        \overline{St}:= \bigcup_{a \in \mathcal{A}}\square^{a}, \text{ for some finite index set } \mathcal{A} \text{ and  cubes } \square^{a}= (\ell, J^{+, a}, J^{-, a}) \text{ with } i \in J^{-, a}.
    \end{equation*}
    Let us define $\overline{T}:= \bigcup_{a \in \mathcal{A}} T(\square^{a})$ (where $T(\cdot)$ was defined in (\ref{eq:Tdef})). Then the strong deformation retractions $\phi^{\square^{a}, i}: \square^{a} \searrow T(\square^{a})$ corresponding to the coordinate $i$ and the $\square^{a}$ cubes of $\overline{St}$ (defined in
    (\ref{eq:phidef})) glue together to a deformation retraction $\overline{\phi}:\overline{St} \searrow \overline{T}$.
\end{claim}

\begin{proof}
    Each $\square^{a}$ cube containing both $\ell$ and $\ell-e_i$ is invariant under the corresponding deformation $\phi^{\square^{a}, i}$ (cf. (\ref{eq:phidef})) and they agree on the intersections, as the formula (\ref{eq:varphidef}) only depends on the points $x$ and $\ell$.
\end{proof}

\begin{cor} \label{rem:phiextended}
     By (\ref{eq:Stesell-Ei}), for every integers $n\geq 2, 1 \leq m \leq \#W_n$, the space $St_{\ell^m}$ can be described as the union of closed cubes containing both $\ell$ and $\ell-e_i$, for the good direction $i \in \mathcal{I}$, thus, we can apply the previous Claim \ref{cl:defretr} to this case. Therefore, we indeed constructed the desired deformation retraction $St_{\ell^m} \searrow T_{\ell^m}$ for every $n\geq 2, 1 \leq m \leq \#W_n$.
\end{cor}

\end{document}